\documentclass{amsart}
\usepackage[utf8]{inputenc}
\usepackage[T1]{fontenc}
\usepackage{amssymb, amsthm, amsfonts}
\usepackage{mathrsfs}
\usepackage{paralist}
\usepackage{mathtools}
\usepackage[color, all]{xy}
\usepackage[colorlinks,citecolor=blue, linkcolor=blue, urlcolor=blue,  pagebackref=true]{hyperref}
\newenvironment{enumb}{\begin{compactenum}[(a)]}{\end{compactenum}}

\newcommand{\oper}[1]{\operatorname{#1}}
\newcommand{\Gal}{\operatorname{Gal}}

\newcommand{\zet}{\mathbb{Z}}
\newcommand{\abs}[1]{\ensuremath{\lvert #1 \rvert}}
\newcommand{\Aut}{\operatorname{Aut}}
\newcommand{\klmm}[1]{\ensuremath{\left( #1 \right)}}
\newcommand{\klmmset}[1]{\ensuremath{\left\{  #1 \right\} }}
\newcommand{\en}{\mathbb{N}}
\newcommand{\ez}{\mathbb{Z}}
\newcommand{\qu}{\mathbb{Q}}
\newcommand{\er}{\mathbb{R}}

\newcommand{\erz}[1]{\ensuremath{\langle #1 \rangle}}

\newcommand{\ef}{\ensuremath{\mathbb{F}}}
\newcommand{\pe}{\mathbb{P}}

\newcommand{\Disc}{\operatorname{D}}
\newcommand{\prodd}{\prod\limits}
\newcommand{\summ}{\sum\limits}

\newcommand{\cha}{\operatorname{char}}
\newcommand{\comp}{\operatorname{comp}}
\newcommand{\ord}{\operatorname{ord}}
\newcommand{\id}{\operatorname{id}}
\newcommand{\inv}{\ensuremath{^{-1}}}
\newcommand{\lcm}{\operatorname{lcm}}
\newcommand{\AGL}{\operatorname{AGL}}
\newcommand{\ASF}{\operatorname{Cond}}
\newcommand{\Sym}{\operatorname{Sym}}

\newcommand{\nujeh}[2]{\nu_{J(#1)}\left( \left[ #2 \right]\right) }
\newcommand{\Hom}{\operatorname{Hom}}
\newcommand{\Ker}{\operatorname{Ker}}
\newcommand{\lt}{\operatorname{lt}}
\newcommand{\trunc}[1]{\ensuremath{\lfloor #1 \rfloor}}

\newtheorem{thm}{Theorem}[section]
\newtheorem{dfn}[thm]{Definition}
\newtheorem{rem}[thm]{Remark}
\newtheorem{lem}[thm]{Lemma}
\newtheorem{cor}[thm]{Corollary}
\newtheorem{exm}[thm]{Example}
\newtheorem{pro}[thm]{Proposition}

\newcommand{\cardI}{\# I}
\newcommand{\cardJ}{\# J}
\usepackage{verbatim} 

\newcommand{\sigmaN}{N}
\newcommand{\ptrunc}[2]{\ensuremath{\trunc{\frac{#1}{#2}}}}

\newcommand{\disc}{\oper{disc}}
\begin{document}
\title{Counting Frobenius extensions over local function fields}	


\author{J\"urgen Kl\"uners}
\address{University Paderborn, Department of Mathematics,
	Warburger Str. 100, 33098 Paderborn, Germany} 

\email{klueners@math.uni-paderborn.de}

\author{Raphael M\"uller}
\address{University Paderborn, Department of Mathematics,
	Warburger Str. 100, 33098 Paderborn, Germany} 

\email{rmuelle2@math.uni-paderborn.de}


\begin{abstract}
We determine the asymptotic growth of extensions of local function fields of characteristic $p$ counted by discriminant, where the Galois group is a subgroup of the affine group $\AGL_1(p)$.  
More general, we solve the corresponding counting problems for all groups which arise in a tower of a cyclic extension of order $p$ over a cyclic extension of degree $d$ coprime to $p$.
This in particular give answers for certain non-abelian groups including $S_3$, dihedral groups of order $2p$, and many Frobenius groups.
\end{abstract}
\maketitle
	
\section{Introduction}

There is a lot of interest to
ask how many extensions with a given Galois group and bounded discriminant exist, and what is the asymptotic behaviour when this bound goes to infinity.
We would like to count non-normal extensions as well, and therefore we define by abuse of notation a Galois group for non-normal extensions.
For this, let $G\leq S_n$ be a finite transitive  permutation group on $n$ points and let $k$ be a 
field. We will write $\Gal(K/k)=G$ if $K/k$ is a field
extension such that the Galois group of the Galois closure $\hat{K}/k$ viewed as
permutation group on the set of embeddings of $K$ into $\hat{K}$ is permutation-isomorphic to $G$. For example when we consider $S_3\leq S_3$ we count degree 3-extensions, and by looking at $S_3(6)\leq S_6$ we count Galois $S_3$-extensions.

For number fields $k$ and transitive permutation groups $G\leq S_n$, the classical case is to count by the norm of the relative discriminant.
\[
Z(k,G;X) := \# \{ K/k ~:~ \Gal(K/k) = G, \quad N_{k/\qu}(\Disc(K/k)) \leq X \}.
\]
For a given $X\in \er$, we have $Z(k,G;X) < \infty$. When $k$ is a $p$-adic field, the total number of extensions is bounded. The same is true when $k$ is a local function field of characteristic $p$, and when $p$ is coprime to $|G|$.
In all other cases, it is interesting to study  the behaviour of $Z(k,G;X)$ for $X \to \infty$.%

 Gunter Malle proposed in \cite{Mal1}, \cite{Mal2} a  conjecture for the asymptotic behaviour of $Z(k,G;X)$ for number fields and 
finite transitive permutation groups. The Malle conjecture  predicts explicit constants $a(G), b(k,G)$ such that 
\begin{equation}\label{Malle-Vermutung}
	Z(k,G;X) \sim c(k,G) \cdot X^{a(G)} \log(X)^{b(k,G)-1}
\end{equation}
for some constant $c(k,G) >0$.

In this work, we focus on the discriminant distribution over local function fields of characteristic $p$ of some non-abelian groups $G$ satisfying $p\mid |G|$. Let $F$ be a local function field of characteristic $p$, i.e., a Laurent series ring over a finite field $\ef_q$ with $q=p^r$ elements. For $p\nmid d$ let $L/F$ be a $C_d$-extension, where $C_d$ denotes the cyclic group of order $d$. Note that there are only finitely many such $L$. Now we study the case that $\Gal(M/L)=C_p$.   Consider the factorisation
\[X^d-1 = f_1(X)\cdots f_r(X) \in \ef_p[X],  \]
and let $I\subseteq \{1,\ldots,r\}$, and denote by $\ell(I):=\deg(\prod_{i\in I}f_i).$ 
We describe all potential Galois groups $\Gal(M/F)$ arising in those towers.
We denote those groups (see \eqref{groups}) of order $dp^{\ell(I)}$ by $G_{p}(d,I)$.
We study the corresponding discriminant counting functions, and prove the asymptotic behaviour of such functions. 

The main results of this paper are described in the following theorem (see Theorems \ref{pd-points} and \ref{thm:Asymptotik-Zerf}), where we count by bounded discriminant exponent. The first function counts fields of degree $pd$, and the second function counts fields of degree $dp^{\ell(I)}$, i.e. Galois extensions. A precise definition of $Z_{pd}$ and $Z_{dp^{\ell(I)}}$ is given in \eqref{def_counting}.
\begin{thm}
There exist a $(p-1)pd$-periodic function $\beta_I(x)$ and a $(p^{\ell(I)}-1)pd$-periodic function $\tilde\beta_I(x)$  such that
\begin{enumerate}
\item considered as a permutation group over $pd$ points, we get for $x\rightarrow \infty $:
\[  Z_{pd}\klmm{ F,G_p(d,I);x } \sim \beta_I(x) q^{ax}\text{, where }a:=\frac{\ell(I)}{pd}. 
\]
\item considered as a permutation group over $dp^{\ell(I)}$ points, we get for $x\rightarrow \infty $:
	\[  Z_{dp^{\ell(I)}}\klmm{ F,G_p(d,I)};x ) \sim \tilde \beta_I(x)  q^{\tilde a x }\text{, where }\tilde a = \frac{p-1}{pd }\frac{ \ell(I) }{p^{\ell(I) } -1}. \]
\end{enumerate}
\end{thm}
The first part of this theorem was already proved in the PhD thesis \cite{RM} of the second author. The second part was only proved in the special case that $|I|=1$. When $\ell(I)=1$, some special cases are the dihedral groups of order $2p$, or more general Frobenius groups $C_p\rtimes C_d$ for $d\mid (p-1)$. Other groups in this list for $p=2$ are $A_4\cong C_2^2\rtimes C_3$ or $C_2^3\rtimes C_7$. 

We see that the asymptotic behaviour is only depending on the group order $dp^{\ell{(I)}}$. The oscillating function is additionally depending on the chosen $I$.
This means that the subgroups of $\AGL(1,p)\cong C_p\rtimes C_d$ have the same asymptotic behaviour as $C_p$ on $p$ points or $C_p \times C_d$ on $pd$ points.

\subsection{History}
For Malle's conjecture, there has been a lot of activity in the number field situation, and in the case of global function fields when the characteristic is coprime to the group order. Let us focus on the case of positive characteristic.  Wright~\cite{Wr} proved equation \eqref{Malle-Vermutung} 
 for every finite abelian group $A$ and over  every global field $k$ such that $\oper{char}(k) \nmid | A|$, prior to the works of Malle. Ellenberg and Venkatesh \cite{EV} gave a heuristic why Malle`s conjecture should be true in the global function field situation when the group order is coprime to $p$ (tame case). In recent works of Landesman and Levy \cite[Theorem 10.1.10]{LaLe25} and Santens \cite{San26} the conjecture in the tame case is proved over rational function fields assuming that the constant field size is large enough.
 
However,  the exponent $a(G)$ in \eqref{Malle-Vermutung}  is incorrect if the characteristic $p$ divides the group  order. Lagemann~\cite[Prop. 2.4]{La3} proved lower bounds for all elementary-abelian $p$-extensions. For non-cyclic groups, those lower bounds (except possibly for the group $C_2\times C_2$) grow larger than \eqref{Malle-Vermutung} indicates. Potthast \cite{Nicolas} proved the discriminant counting for all elementary abelian $p$-groups in characteristic $p$. The asymptotics of  $C_2\times C_2$ has one log-factor more than a corresponding Malle conjecture would predict. Other interesting results are given in \cite{Gund25,GuSe25}, where the counting (for rational global function fields) is done by the last jump in the upper numbering of the higher ramification groups. In \cite{Gund25}, the surprising result is that the counting function over rational function fields is a rational function. In \cite{GuSe25} Gundlach and Seguin, deal with some cases of non-abelian $p$-groups over rational global function fields. Already Lagemann \cite{La3,La2} proved the corresponding asymptotics for counting abelian $p$-groups by conductor over global function fields of characteristic $p$.

It is well known that there are infinitely many extensions for a given $p$-group $G$ over a local function field of characteristic $p$. Therefore, it is an interesting problem to ask for the asymptotic behaviour of the number of such extensions. 
Lagemann solved the case of counting abelian $p$-groups over local function fields by discriminant in \cite{La1}.  The corresponding count over local function fields by conductor is given in \cite{KLUNERS2019}. In the PhD-thesis \cite{RM} of the second author there are asymptotic results for (generalized) Heisenberg groups of $p$-power order over local function fields of characteristic $p$.

 \subsection{Structure of the paper}
 We will consider this problem in the situation of a local function field $F$ of characteristic $p$ and a tower of fields $M/L/F$ such that $\Gal(M/L)\cong C_p$ and 
$ \Gal(L/F) \cong C_d$ is cyclic with $p\nmid d$. 
By Artin-Schreier Theory, the $C_p$-extensions of $L$ can be described as $\ef_p$-subspaces of the vector space $L/\wp(L)$, 
where $\wp(x)=x^p-x$ is the Artin-Schreier operator.
In this situation, $L$ and $L/\wp(L)$ have a nice $\ef_p[C_d]$-module structure which we can use to describe the discriminant counting function for all occurring Galois groups in this setting. 

In Chapter 2, we collect known results about tamely ramified extensions and Artin-Schreier theory.  In Chapter 3, we study $\ef_p[C_d]$-modules and we show in Proposition \ref{Gal-Disc} how to compute the Galois group and the discriminant, when we know the Artin-Schreier generator.

In Chapter 4, we prove the asymptotics when we consider our groups $G_p(d,I)\leq S_{pd}$ on $pd$ points. In Chapter 3, we showed that we have to choose our Artin-Schreier generators from some (generalized) eigenspaces.  Since there are only finitely many $C_d$-extensions $L$, we count all extensions over all possible fixed $L$. In the end, the results follows by a finite summation.

In Chapter 5, we deal with the Galois situation, and the strategy is similar. The major problem is to control the discriminant of the elementary abelian $p$-extension of the Galois closure over $L$. Here we have the same type of problems as in the elementary abelian case which was solved in \cite{Nicolas}, and we use the methods described in that paper.

\subsection{Acknowledgements}
This work was supported by the Deutsche Forschungsgemeinschaft (DFG, German Research Foundation) — Project-ID 491392403 — TRR 358 (project A4). The authors thank Nicolas Potthast, and B\'eranger Seguin for helpful discussions and feedback.

\section{Preliminaries}
\subsection{Tamely ramified cyclic extensions}

There is a well-known classification of the finitely many (at most) tamely ramified $C_d$-extensions of a local field $F$.

\smallskip
\begin{thm}[Classification of tamely ramified extensions]\label{thm:ClassTame}
	Let $F$ be a local function field with residue class field $\ef_q$ and prime element $\pi_F$. Let $L/F$ be an at most tamely ramified extension with ramification index $e=e_{L/F}$ and inertia degree $f=f_{L/F}$. Let $\klmm{\ef_{q^f}}^\times = \erz{\omega}$ and $g := \gcd (e,q^f-1)$.
	\begin{enumb}
		\item Then $L$ is conjugate to exactly one field $F(\omega, \sqrt[e]{\omega^r \pi_F})$ where $0 \leq r < g$.
		\item $L/F$ is a Galois extension if and only if $ e \mid (q^f-1) $ and 
		$e \mid r(q-1)$. \\
		Let $\pi_L:= \sqrt[e]{\omega^r \pi_F}$. If  $L/F$ is Galois, then  $\Gal(L/F)=\langle\sigma_1,\sigma_2\rangle$ with
		\begin{align*}
			\sigma_1(\omega) & = \omega, & \    \sigma_1(\pi_L) &= \omega^l \cdot \pi_L, \\ 
			\sigma_2(\omega)& =\omega^q, & \sigma_2(\pi_L)& =\omega^k \cdot \pi_L, 
		\end{align*}
		where	$k=\frac{r(q-1)}{e}$ and $l=\frac{q^f-1}{e}$. The Galois group has the finite presentation 
		\[ \Gal(L/F) = \erz{\sigma_1, \sigma_2 ~|~ \sigma_1^e=1, \ \sigma_1^r=\sigma_2^f, \ \sigma_1^{\sigma_2}=\sigma_1^q} . \] 
		\item The extension $L/F$ is abelian if and only if $e \mid( q-1)$.
		
		It is moreover cyclic if and only if $e\mid (q-1)$ and $\gcd(e,f,r)=1$. In this case we have a generator $\sigma = \sigma_1^N\sigma_2$, where 
		\begin{equation} \label{sigmaN}
			\sigmaN := \begin{cases}  \prodd_{\ell \in \pe, \ell\mid e, \ell \nmid r} \ell, & \gcd(e,r)> 1 \\ 0, & \gcd(e,r) = 1. \end{cases} 
		\end{equation} 
		In particular, the generator $\sigma $ of $\Gal(L/F)$ has the property
		\begin{equation}\label{sigmaprop}
		\sigma(\omega ) = \omega^q , \quad \sigma (\pi_L)  = \omega^{k+Nl} \cdot \pi_L.		
		\end{equation}
	\end{enumb}
\end{thm}
\begin{proof} (a) and (b) are proven in \cite[Chapter 16, p. 249ff]{Ha}. 
	
	In (c) note that  $\ord(\sigma_1)=e$ and $\Gal(L/F)$  is abelian if and only if \[ \sigma_1 = \sigma_1^{\sigma_2} = \sigma_1^q  \iff \sigma_1^{q-1} = \id \iff \ord(\sigma_1) = e \mid (q-1).\] 
	
	If   $\Gal(L/F) = \erz{\sigma_1,\sigma_2}$ is abelian, we have 
	\[ \exp(\Gal(L/F)) = \lcm \klmm{  \ord(\sigma_1 ), \ \ord(\sigma_2) }.\]  
	Concerning $\sigma_2$, we use
	$\ord(\sigma_1)=e$ to get 
	\[
	\ord(\sigma_2) \overset{(b)}{=} f \cdot \ord(\sigma_1^r) = f \cdot \frac{e}{\gcd(e,r)}. 
	\] 
	Thus, $\Gal(L/F)$ is cyclic if and only if $\Gal(L/F)$ is abelian and 
	\begin{align*} 
		d = \exp(\Gal(L/F)) \iff &  e\cdot f = \lcm(\ord(\sigma_1), \ord(\sigma_2)) = \lcm( e, \ f \frac{e}{\gcd(e,r)}) \\
		\iff & e\cdot f = \frac{e}{\gcd(e,r)} \cdot \lcm \klmm{  \gcd(e,r) , \ f }  \\
		\overset{\cdot \ \frac{\gcd(e,r)}{e}  }{\iff} & \gcd(e,r)\cdot f = \lcm(\gcd(e,r),\ f)\\
		  \iff& \gcd(e,r,f)=1.
	\end{align*} 
	Now we can assume that $L/F$ is cyclic. 
	We have $\ord(\sigma_1)=e$ and $\ord(\sigma_2)= \frac{ef}{\gcd(e,r)}$.
	Then $\sigma=\sigma_2$ is trivially a generator for $\gcd(e,r)=1$ which implies $N=0$.

	 In the case $\gcd(e,r)>1$ let $\ell$ be a prime  with 
	\begin{itemize}
		\item $\ell\nmid e$: Then $\nu_\ell(ef)=\nu_\ell(f)=\nu_\ell(\ord(\sigma_2))=\nu_\ell(\ord(\sigma_1^N\sigma_2))$.
		\item $\ell\mid e,\ell\nmid r$: This implies $\ell\mid N$, $\ell\nmid \gcd(e,r)$ and therefore \\$\nu_\ell \klmm{ \ord(\sigma_2)}=\nu_\ell(ef)$. Since $\nu_\ell(\ord(\sigma_1^N))<\nu_\ell(ef)$ we get that\\ $\nu_\ell \klmm{ \ord(\sigma_1^N\sigma_2)}=\nu_\ell(\ord(\sigma_2))=\nu_\ell(ef)$ as wanted.
	\item $\ell\mid e,\ell\mid r$: This implies  $\ell\nmid f$, and $\ell\nmid N$.	We therefore get that \\
	$\nu_{\ell}\klmm{  \ord(\sigma_2 )}  < \nu_\ell(ef) = \nu_\ell(e)$ and 
	\[ \nu_\ell(\ord(\sigma_1^N \sigma_2))=\nu_\ell(\ord(\sigma_1 \sigma_2)) = \nu_\ell(\ord(\sigma_1)) = \nu_\ell(e) = \nu_\ell(ef).\]
\end{itemize}	
	Hence, the element $\sigma := \sigma_1^\sigmaN \sigma_2 $ for  $ \sigmaN $ defined in \eqref{sigmaN}
	has order $ef$ and  
		is a generator of the Galois group.
	Finally, for the last assertion we have
	\[ \sigma_2(\omega) = \omega^q = \sigma_1^\sigmaN \sigma_2(\omega) \ \text{ and } \ 
	\sigma_2(\pi_L) = \omega^k\pi_L\] and 
	\[ \sigma_1^\sigmaN \sigma_2(\pi_L)
	= \sigma_1^\sigmaN( \omega^k \pi_L) 
	= \omega^k \omega^{\sigmaN l} \pi_L. \qedhere \]
\end{proof}

 \subsection{Artin-Schreier Theory}

For the moment, let  $L$ be any field with $\cha(L)=p$. We define the Artin-Schreier operator 
\[ \wp\colon L \to L, \quad x \mapsto x^p - x. \] This map $\wp$  is $\ef_p$-linear  with kernel 
$ \ker(\wp) = \ef_p$.
We moreover define $J(L) := L / \wp(L)$ as the cokernel of $\wp$.
We will write $\theta_a $ for a solution of $X^p-X-a$. Note the factorisation $X^p-X-a = \prod_{\lambda \in \ef_p} (X-\theta_a+\lambda)$.

By Artin-Schreier theory, $J(L)$ is a parameterising space for all elementary $p$-extensions of $L$, e.g. see
\cite[p.~294f]{Ne}:
\begin{thm}[Artin-Schreier theory] 
	Let $L$ be a field with $\cha(L)=p$. 
	\begin{enumb}
		\item There is a $1:1$-correspondence 
		\begin{align*}  \Delta \colon \{  
			\ef_p\text{-subspaces } U \leq L / \wp(L)\} & \longrightarrow \{ \text{p-elementary field extensions } L'/L \} \\ U & \longmapsto L(\wp\inv(U)
			). \end{align*}
		\item  For $U \leq L/\wp(L)$ and $L' = L \klmm{ \wp\inv(U) }$ there is a canonical isomorphism $$ U \cong \Hom \klmm{  \Gal(L'/L) , \ef_p }, \quad a +\wp(F) \mapsto \chi_a \quad \text{via }\chi_a(\sigma) = (\sigma -1 )(\theta_a).$$ 
		\item	Let $U\leq L/\wp(L)$ be finite and $(a_1 +\wp(L),\ldots, a_r +\wp(L))$ be an $\ef_p$-basis of $U$. Then 
		the Galois group 	$\Gal \klmm{ L\klmm{ \wp\inv(U) } /L } \cong \klmm{ C_p }^r $ is generated by the automorphisms $\sigma_i$ with
		\[  \sigma_i( \theta_{a_j}) = \theta_{a_j} + \delta_{i,j} \quad \text{for }\quad 1\leq i \leq r,  \ 1\leq j \leq r, \]
		where $\delta_{i,j}$ is the Kronecker-Delta.
	\end{enumb} 
\end{thm}

We apply Artin-Schreier theory to a field $L$ which is given as a cyclic degree $d$ extension of a local function field $F$. Let us therefore assume that $L$ is given as the Laurent series ring $\ef_{q^f}((\pi_L))$, where $\pi_L$ is a prime element and the constant field has $q^f$ elements.

Choosing $\omega_0 \in \ef_{q^f} \setminus \wp(\ef_{q^f})$ we get 
a nice system of representatives of $L/\wp(L)$ by, e.g. see \cite[Prop 5.2 (b)]{Nicolas}
\begin{equation*}
	R_L(\pi_L, \omega_0) := \left\{  \mu_0 \omega_0 + \summ_{ \substack{ i< 0 \\ p\nmid i }} \lambda_i \pi_L^i ~|~ \mu_0 \in \ef_p, \ \lambda_i \in \ef_{q^f} \right\}.
\end{equation*}
Usually, we simply write $R_L$. 
For any $\beta \in L \setminus \wp(L) $ we define 
\[ 
\nujeh{L}\beta := \max_{ y \in L} \nu_L\klmm{ \beta + \wp(y)} \leq 0.
\]
Independently of the choices of $\pi_L$ and $\omega_0$, the representative system $R_L$ has the following properties:
\begin{lem} $R_L$ is an  $\ef_p$-vector space with $\nu_L(\alpha) \leq 0 $ for all $0 \neq \alpha \in R_L$ and \[\nu_{J(L)}([\alpha]) = \nu_L(\alpha)\text{  for }0 \neq \alpha \in R_L.\] 
\end{lem}
This lemma is useful to determine the conductor and the discriminant of an Artin-Schreier extension.

\begin{dfn}
	For $\alpha\in L\setminus \wp(L)$ we define the conductor exponent via
\[ 
\ASF_L(\alpha) := \begin{cases} 0, & \nujeh{L} \alpha = 0 \\ \abs{\nujeh{L} \alpha } +1 , & \nujeh{L}\alpha <0 . \end{cases}
\]
\end{dfn}
For such an $\alpha$ the conductor coincides with the usual conductor for cyclic extensions and therefore we get the following formula for the discriminant exponent.
\begin{equation*}
\disc( L' /L) =  (p-1) \ASF_L(\alpha).
\end{equation*}
For a subgroup $U\leq L/\wp(L)$, we use the abelian conductor discriminant formula (see \cite[Thm. 7.15]{Iw}): Let $L' := L \klmm{ \wp\inv(U) }$. 
Let $0\neq \chi_a  \in  \Hom \klmm{  \Gal(L'/L) , \ef_p}$,
then $\wp \inv(a) \leq \operatorname{Fix}\klmm{ \Ker(\chi_a)}$ by definition, hence by comparing degrees, we get 
	\[ 
	L'_{\chi_a} = \operatorname{Fix}(\Ker(\chi_a)) = L\klmm{ \wp\inv(a) }
	\] and 
 \begin{equation*}
 	\mathfrak f (\chi_a ) = \mathfrak f \klmm{ F(\wp\inv (a)) / F} = \ASF_{L}\klmm{ a}.
\end{equation*}
Thus, for the discriminant exponent of an Artin-Schreier extension with respect to $U\leq L/\wp(L)$, we have
\begin{equation} \label{ElementaryAbelian-Disc}
	\disc\klmm{ L'/L } = \disc\klmm{ L(\wp\inv(U)) /L  }  = \sum_{0\neq a \in U} \ASF_L(a).
	\end{equation}

Finally, for  $\sigma \in \Aut(L/F)$ holds
\begin{equation}\label{wp-sigma-vertauschen}
\sigma \circ \wp = \wp \circ\sigma. 
	\end{equation}

\section{Galois Groups and Discriminants}

Let $L/F$ be a cyclic degree $d$-extension. In the following we would like to study the Galois group and the discriminant of a $C_p$-extension $L(\theta_\alpha)/L$. For $\Gal(L/F)=H=\langle\sigma\rangle$, we can understand the Galois group of $L(\theta_\alpha)/F$ by studying the $\ef_p[H]$-module structure of the module generated by $\alpha$. In a second part we explain how to compute the discriminant.

From now on, we fix $F=\ef_q((t))$ and \[ L = \ef_{q^f}(( \sqrt[e]{\omega^r t})) =: \ef_{q^f}((\pi_L))\]
 for some $(q^f-1)$-st root of unity $\omega \in \ef_{q^f}$, where $p\nmid e\cdot f$ and  $H:=\Gal(L/F) = \erz{\sigma}$ is cyclic with $\sigma$ as in Theorem~\ref{thm:ClassTame}(c).
In particular, $\sigma$ has the properties $\sigma|_{\ef_q} = \id$, $\sigma\mid_{\ef_{q^f}}$ acts as Frobenius automorphism via $\sigma(\omega) = \omega^q$ and $\sigma(\pi_L) \in \ef_{q^f} \cdot \pi_L$.
We write  \[ V_z := \ef_{q^f}\cdot \pi_L^z \quad \text{for all}\quad z \in \zet.\] 
Clearly, 
\( \sigma(V_z) = V_z \) and thus it forms a $H$-invariant $\ef_{q^f}$-subspace of $L$.

\begin{figure}[h]\label{diagramm}
	\begin{align*}
		\xymatrix{
			& L_\alpha:=\oper{Spl}_F\klmm{ L(\theta_\alpha)}  \ar@{-}[dl] \ar@{-}[dr]  \ar@{..}[dd]^{ \klmm{ C_p }^\ell} \\
			L(\theta_\alpha) \ar@{-}[dr]^{C_p} &  & L(\theta_{\sigma^i(\alpha)})  \ar@{-}[dl]_{C_p} \\
			& L \ar@{-}[d]^{H = \erz{\sigma}} &  & \\
			& F &
		}
	\end{align*}
	\caption{Field diagram}
\end{figure}

It turns out that the $\ef_p[H]$-module action on $\oplus_{z\in \ez} V_z$ is periodic with period $e$.
\begin{thm}\label{thm:Fq-L-Structure} \ 
		\begin{enumb}
	\item There is a primitive $e$-th root of unity $\zeta_e \in \ef_q$ so that $\sigma|_{V_z}$ has 
		minimal polynomial $X^f - \zeta_e^z $ for all $ z \in \zet $. In particular, 
		$ V_{z + e} \cong V_z$ as $\ef_q[H]$-modules for all $z \in \zet $.
	\item We have  $V_z \oplus V_{z + 1} \oplus  \ldots \oplus V_{z + e - 1} \cong
			\klmm{ \ef_p[H]}^{[\ef_q : \ef_p ]}$ as $\ef_p[H]$-modules.
		\end{enumb}
	\end{thm}
\begin{proof} 
	By \eqref{sigmaprop}, $\sigma$ acts on the constant field as Frobenius automorphism, which implies that $\tau:=\sigma^f$ is the identity on $\ef_{q^f}$. 
	Using $\ord(\sigma)=ef$ we get that $\ord(\tau)=e$. Now we get that 
	$\tau(\pi_L)=\omega^x\pi_L$ and therefore using $\tau(\omega)=\omega$ we get
	that $\tau^e(\pi_L)=\omega^{xe}\pi_L=\pi_L$. This implies that $\omega^x = \zeta_e$, where $\zeta_e$ is an $e$-th root of unity, which must be primitive, since $\ord(\tau)=e$. 

	Using $\tau(\omega) = \omega$, we  obtain $ \tau(\omega^k \pi_L^n) = \zeta_e^n \omega^k \pi_L^n $, hence 
	\begin{equation} \label{eq:zeta-en}
		 (\sigma^f - \zeta_e^n)(V_n) = 0. \end{equation}
Now \eqref{eq:zeta-en} implies that 
$V_z \oplus V_{z + 1} \oplus  \ldots \oplus V_{z + e - 1} \cong \ef_q[X]/(X^d-1) \cong 
			\ef_q[H]$   as $\ef_q[H]$-modules.

For (b) note finally that $\ef_{q}[H] \cong \ef_p[H]^{[\ef_q : \ef_p]}$ via scalar extension.
\end{proof}


Since $p\nmid d$ and therefore $p\nmid f$ we can choose 
\begin{equation}\label{omega0}
	\omega_0 \in \ef_q^\times \setminus \wp(\ef_{q^f}).
\end{equation} 
Such an element exists since $\ef_q \subseteq \wp(\ef_{q^f})$ implies that the unique $C_p$-extension of $\ef_q$ was contained in $\ef_{q^f}$, which is impossible for $p \nmid f$.

Concerning the $\ef_p[H]$-module structure on $L/\wp(L)$ we study for $z < 0 , \ z \in \zet $
the $\ef_p[H]$-modules
\begin{equation}\label{Wz-Vi-Dec}
W_z := \sum_{ \substack{i= ep\cdot z \\p\nmid i } }^{ep\cdot(z+1) -1} V_i .
\end{equation}

The submodules $W_z$ and $ \ef_p \omega_0 \cong \ef_{q^f} /\wp\klmm{ \ef_{q^f} }$ determine the 
$\ef_p[H]$-structure of $R_L$. Note  that we have the periodicity $W_{z} \cong W_{k}$ as 
$\ef_p[H]$-modules for all $k \in \zet$ due to  $V_{i+k pe} \cong V_i$ for all $i \in \zet$.

\smallskip
\begin{pro}\label{RL-Iso}
	Let $\pi_L$ as in Theorem~\ref{thm:ClassTame} and $\omega_0 \in \ef_q^\times \setminus \wp(\ef_{q^f})$ as in \eqref{omega0}. Then we have an $\ef_p[H]$-module isomorphism 
	\begin{equation*}
		R_L (\pi_L, \omega_0) = \ef_p \omega_0 \oplus \bigoplus_{ z <  0} W_z  \overset{\sim}\longrightarrow L/\wp(L) , \quad x \longmapsto x + \wp(L),
	\end{equation*}
where 
\begin{equation*}
	\ef_p \omega_0 \oplus \bigoplus_{ z <  0} W_z \cong 	\ef_p \omega_0 \oplus \bigoplus_{z < 0 } 
	\ef_p[H]^{(p-1) [\mathbb F_q :\mathbb F_p ]} \quad \text{(as } \ef_p[H]\text{-modules)}.
\end{equation*}
\end{pro}
\begin{proof}
	For $i< 0$ we get that $\sigma(\lambda_i \pi_L^i) = \lambda_i^q (\omega^{k+Nl})^i \pi_L^i = \mu_i \pi_L^i \in R_L(\pi_L,\omega_0)$ for some $\mu_i\in\ef_{q^f}$. Since $\omega_0\in \ef_q$ we get  $\sigma(\omega_0)=\omega_0$. Thus, the left hand side is an $\ef_p[H]$-module.
	
	It is well-known that the map is bijective and  $\ef_p$-linear, $\sigma \circ \wp (x) \overset{\eqref{wp-sigma-vertauschen}}{=} \wp \circ \sigma(x)$ and $\sigma(\wp(L)) = \wp(L)$ which establishes the isomorphism in \eqref{RL-Iso}.
		
	Using Theorem~\ref{thm:Fq-L-Structure}, we get a chain of $\ef_p[H]$-isomorphisms:
	\begin{equation}\label{Decomposition-Wn} 
		W_z	 \overset{\eqref{Wz-Vi-Dec}}{=} \bigoplus_{ \substack{ i= ep \cdot z  \\ p \nmid i } }^{ep \cdot (z + 1) -1} V_i \underset{\text{Thm.~\ref{thm:Fq-L-Structure}(a) }}\cong \bigoplus_{k=1}^e V_k^{p-1} 
		\underset{\text{Thm.~\ref{thm:Fq-L-Structure}(b) }}\cong \klmm{ \ef_p[H]}^{(p-1) \cdot [\ef_q : \ef_p]}.
		\qedhere
	\end{equation}  
	\end{proof}	
In the following, we write $R_L := R_L(\pi_L,\omega_0)$.
The aforementioned isomorphism determines the $\ef_p[H]$-module-structure $R_L \cong L/\wp(L)$ parametrising all $p$-elementary abelian extensions of $L$.  This 
enables us to both control the Galois group and the discriminant of simple Artin-Schreier extensions 
$L(\wp\inv ( \alpha))$ for $\alpha \in L$.

Thus, it is sufficient to consider the $\ef_p[H]$-module structure of  $\ef_p[H]$ to describe the 
possible Galois groups. 

Towards this goal, we consider the prime factorisation 
\begin{equation*}
	X^d -1 = f_1 \cdots f_r \in \ef_p[X], \qquad f_1(X) = X-1.
\end{equation*}
We fix the choice $f_1(X):= X-1$ since we will obtain a slightly different counting formula for 
the factor $X-1$ which  corresponds to the $C_p \times H$-extensions.

\smallskip
\begin{dfn}
Let $\erz \sigma \cong C_d$ and $V$ be an $\ef_p[H]$-module. Let $I\subseteq \{ 1,\ldots, r\}$. We define 
\[
Y_{f_i} (V) := \Ker(f_i(\sigma)), \qquad Y_I(V) :=  \prod_{i \in I} Y_{f_i} (V) 
\]
\end{dfn}
Equivalently, $Y_I(V)$ can be described as the kernel $g_I(\sigma )$ for $g_I := \prod_{i\in I} f_i$. 
\smallskip 
\begin{exm}
\begin{enumb}
	\item 	For $V:= \ef_p[H]$ and $i \in\{ 1,\ldots, r\}$ we get using $\ell(i) =\deg(f_i)$ 
	\begin{equation}\label{def_Mi-Dec_Gruppenring}
M_i:=	Y_{f_i}(\ef_p[H]) \cong \ef_p[X]/ (f_i) \cong \ef_p^{\ell(i)}, 
	\end{equation} 
\begin{equation}\label{def_MI}
	M_I := Y_I(\ef_p[H])\cong  \ef_p[X]/ (  \prod_{i\in I} f_i )
\end{equation}
as $\ef_p$-modules.
	 By the Chinese Remainder Theorem, we have 
	\[ 
	\ef_p[H] \cong \ef_p[X]/(X^d-1) \cong \prod_{i=1}^r \ef_p[X] / (f_i) \cong \bigoplus_{i=1}^r M_i.
	\]
	The $M_i$ are $\ef_p[H]$-modules via $\sigma m:= X\cdot m$, and we will identify $M_I$ with $\{ g(X) \in \ef_p[X] ~|~ \deg(g) < \ell(I) \}$.
	\item Let $z \in \zet$ and $V:= W_z$ as in \eqref{Wz-Vi-Dec}. Then
	\begin{equation*}
	Y_{f_i}(W_z) \overset{\eqref{Decomposition-Wn}}{\cong} Y_{f_i}\klmm{ \ef_p[H]^{(p-1)[\ef_q :\ef_p]} } \overset{\eqref{def_Mi-Dec_Gruppenring}}{\cong} \klmm{ \ef_p^{\ell(i)} }^{ (p-1)[\ef_q :\ef_p]}
	= \ef_p^{\ell(i) (p-1)[\ef_q :\ef_p]}.
	\end{equation*}
	In particular,  
	\begin{equation} \label{Anzahl-YI} 
		| Y_{f_i}(W_z)| = q^{\ell(i)(p-1)}. 
	\end{equation}
\item 
For $1\leq i \leq r$ and $I\subseteq \{ 1,\ldots, r \}$ we will write 
\begin{equation}\label{Covidl}
	Y_i := 	Y_{f_i}(R_L)  = \Ker(f_i(\sigma))  , \qquad Y_I := \bigoplus_{i \in I} Y_i
\end{equation} 
and obtain a decomposition 
\begin{equation}\label{Decomposition-RL}
	R_L = Y_1 \oplus \ldots \oplus Y_r.
\end{equation}
Note furthermore that 
\begin{equation} 
	\label{Yi-dec}
	Y_{i} = Y_{f_i} \klmm{ \ef_p \omega_0 } \oplus \bigoplus_{z < 0} Y_{f_i}\klmm{ W_z } 
\end{equation}	
as all these subspaces are $\ef_p[H]$-invariant, and thus 
\[ 
Y_i = Y_{f_i}(R_L ) = Y_{f_i} \klmm{\ef_p \omega_0\bigoplus_{z < 0} W_z  } =  Y_{f_i} \klmm{ \ef_p \omega_0 } \oplus \bigoplus_{z < 0} Y_{f_i}\klmm{ W_z } .
\]
	\end{enumb}
\end{exm}
For $I\subseteq \{ 1,\ldots, r \}$ we define 
\begin{equation}\label{groups}
G_p (d,I) :=  M_I \rtimes C_d  
\end{equation} 
with $M_I \leq \ef_p[H]$ as in \eqref{def_MI}. For convenience, we   define $ G_p(d,i) := G_p(d,\{ i\} ). $

We will write $ \ell(i) := \deg(f_i)$ and  $\ell(I) :=
 \deg\klmm{ \phi(X)  } = \summ_{i\in I} \deg(f_i) $.

Note that $G_p(d,I) \cong C_p^{\ell(I)} \rtimes C_d$, and if $\deg(f_i)=1$, we have 
$G_p(d,i) \cong C_p\rtimes C_d \leq \AGL_1(p)$.



Now we show that the module structure determines the Galois group.

\begin{pro} \label{Gal-Disc}
	Let $0\neq \alpha = \alpha_1+\ldots + \alpha_r \in Y_1 \oplus \ldots \oplus Y_r =R_L$. 
	Let $I := \left\{ i \in \{1,\ldots , r\} ~|~ \alpha_i \neq 0 \right\}$. 
 	Then we have 
 	$\erz{\alpha}_{\ef_p[H]} \cong M_I$  as $\ef_p[H]$-modules and 
	\[ 
	\Gal(L(\wp\inv(\alpha)) / F ) \cong G_p(d,I) \leq S_{pd}.
	\]
The discriminant equals
\begin{equation}\label{pd-disc} \disc( L \klmm{ \wp\inv(\alpha) }/F) = pd-f + (p-1)f |\nu_L(\alpha)|.
\end{equation}
\end{pro}

\begin{proof} 
	The splitting field $L_\alpha$ (see figure on page \pageref{diagramm})  of $L(\wp\inv(\alpha))/F$ is 
	\begin{align*}  L_\alpha & = 
		 L\klmm{ \wp\inv \klmm{ \erz{\alpha}_{\ef_p[H]}} }, \text{ where } 
		 \erz{\alpha}_{\ef_p[H]} = \left\{ g(\sigma)(\alpha)   ~:~  
		 g(X) \in \ef_p[X] \right\}. 
		\end{align*}
	For $g(X) \in \ef_p[X]$, we have 
	\[ 
	g(\sigma)(\alpha) = 0 \iff g(\sigma)(\alpha_i) = 0 \ \text{for all } i \in I \iff f_i(X) \mid g(X) \text{ for all } i \in I,
	\]
	hence $\erz{\alpha}_{\ef_p[H]} \cong \ef_p[X]/(\prod_{i\in I} f_i) = M_I$. 
		
	Finally, $\Gal\klmm{L_\alpha/L }$ is a normal subgroup of 
	$\Gal\klmm{ L_\alpha/F}$, since its fixed field $L $ is Galois over $F$. With $\gcd(p,d)=1$ and the Theorem of Schur-Zassenhaus (see \cite[Hauptsatz I.18.1]{Hu}) we have that $\Gal\klmm{ L_\alpha / F}$ is a 
	semi-direct product
	with action determined by $\sigma \cdot g(\alpha) := g(\sigma) ( \alpha)$ 
	which concludes the proof for the Galois group. For the discriminant note
	\begin{align*} \disc( L \klmm{ \wp\inv(\alpha) }/F) & = p\disc(L/F) + (p-1) f (|\nu_L(\alpha)| +1 ) \notag \\
		& = pf(e-1) + (p-1) f( |\nu_L(\alpha)| + 1 ) \\
		& = pd-f + (p-1)f |\nu_L(\alpha)|.\qedhere
	\end{align*}
\end{proof}

Now we compute discriminant of the splitting field.
Let $\alpha = \alpha_1+ \ldots + \alpha_r \in Y_1\oplus \ldots \oplus Y_r$. 
For $\nu_i := \nu_L(\alpha_i)$, write  $\alpha_i = \summ_{j= \nu_i}^0 \eta_{i,j} \pi_L^j, \eta_{i,j}\in \ef_{q^f}$ and set 
\[ 
\lt(\alpha_i) : = \eta_{i, \nu_i} \pi_L^{ \nu_i}.
\]
Note that $\lt(\alpha_i) \in V_{ \nu_i}$ which is $\sigma$-invariant. Hence we get 
\begin{equation*}
g(\sigma)(\alpha) = 0 \iff g(\sigma)(\alpha_i) = 0 \text{ for all }i \in I \iff g(\sigma)(\lt(\alpha_i)) = 0 \text{ for all } i \in I,
\end{equation*} as the two last statements are both equivalent to $f_i \mid g$ for all $i \in I$.
Thus, we have 
\begin{equation*}
	\nu_L \klmm{g(\sigma)(\alpha_i) } = \begin{cases} 0, & f_i \mid g \\ \nu_L(\alpha_i), & \text{else,}
		\end{cases} 
\end{equation*}
and 
\begin{equation}\label{val-lt}
	\nu_L\klmm{ g(\sigma)(\alpha) } = 
	\min \left\{  \nu_L(\alpha_i) ~:~ 1\leq i\leq r, \ f_i \nmid g \right\}.
\end{equation}
Thus,  an ordering of the valuations of $\alpha_i$ basically determines the field discriminant.

\smallskip
\begin{pro}\label{Disc-Thm}
Let $\alpha = \alpha_1+\ldots + \alpha_r \in R_L$, let $I:= \{ 1\leq i \leq r ~:~ \alpha_i \neq 0\}$ and let $L_\alpha$ be the splitting field of $L\klmm{\wp\inv(\alpha) }/F$. 
\begin{enumb}
 \item The splitting field $L_\alpha$ is the composite field of the according splitting fields $L_{\alpha_i}$  of $\alpha_i$, i.e. $L_\alpha = \prod_{i\in I} L_{\alpha_i}$.
\item Let $\ell(i) :=  \deg(f_i)$ 
and 
 $\tau \colon \{ 1,\ldots, \cardI \} \to I$ be a bijection such that \\
$\ASF_L\klmm{ \alpha(\tau(1)) }   \leq \ldots \leq \ASF_L\klmm{ \alpha(\tau(\cardI)) }$.  Then we get the discriminant exponent
\begin{align*}
	\disc( L_\alpha / L) & = 
	\summ_{j=1}^{ \cardI }  p^{\ell(\tau(1))} \cdots p^{\ell(\tau(j-1))} \cdot \klmm{ p^{\ell(\tau(j))}  -1} \ASF_L(\alpha_{\tau(j)}).
\end{align*}
\item Writing $R_j := \ell(\tau(1)) + \ldots + \ell(\tau(j))$, we have 
\begin{align}\label{eq:disc-Ltheta-alpha-Spl2}
\disc( L_\alpha / F ) & = 
p^{\ell(I)} f (e-1) + \summ_{j=1}^{ \cardI } f \klmm{ p^{R_j}  - p^{R_{j-1}}} \ASF_L(\alpha_{\tau(j)}).
\end{align}
\end{enumb}
\end{pro}
\begin{proof}	
As in Proposition~\ref{Gal-Disc} we have
$L_\alpha = L\klmm{ \wp\inv \klmm{ \erz{\alpha}_{\ef_p[ H]} }}$. 

On the one hand, we have $\erz{ \alpha_1 + \ldots + \alpha_r}_{\ef_p[H]} \subseteq 
\erz{ \alpha_1}_{\ef_p[H]} + \ldots + \erz{ \alpha_r}_{\ef_p[H]}$ which shows one containment in (a).
Moreover, with $h_j := \prod\limits_{i\neq j} f_i$ we have $h_j(\alpha_i)=0$ for $i\neq j$ and $h_j(\alpha_j) \neq 0$, hence 
\[ 
h_i(\alpha) =  h_i(\alpha_i) \in \erz{\alpha}_{\ef_p[H]} \quad \text{ for all i }\in I,
\] thus $L_{\alpha_i} \subseteq L_\alpha$ which concludes the second containment and the equality in (a) is established.

For (b), we have  $\disc(L_\alpha/L ) = \summ_{ 0\neq \beta \in \erz{\alpha}_{\ef_p[H]}}  \ASF_L(\beta)$ by \eqref{ElementaryAbelian-Disc} and for every 
\[ \beta = g(\sigma)(\alpha) =  g(\sigma)(\alpha_1) + \ldots + g(\sigma)(\alpha_r) \in \erz{\alpha}_{\ef_p[H]}, \] we have
\[ \nu_L(\beta ) \overset{\eqref{val-lt}}{=} 
\nu_L\klmm{ \alpha_{\tau(\omega_g)} }, \quad \omega_g = \max \{  1\leq j \leq \cardI ~:~ g(\sigma)(\alpha_{\tau(j)}) \neq 0 \},
\] where we use the ordering of the valuations. 

It remains to compute the number $N_j$ of elements $g(\sigma)\in \ef_p[H]$ such that $\omega_g = j$ for $1\leq j \leq \cardI $. 
Using the $M_i$ defined in \eqref{def_Mi-Dec_Gruppenring} this number is simply  
\begin{equation}\label{number-MJ}
\delta_j(\tau,I) := |M_{\tau(1)}| \cdots |M_{\tau(j-1)}| \cdot \klmm{ | M_{\tau(j)} | -1 }=p^{\ell(\tau(1))} \cdots p^{\ell(\tau(j-1))}\cdot \klmm{ p^{\ell(\tau(j))}  -1} , 
\end{equation} as this corresponds to the polynomials divisible by $f_{\tau(i)}$ for $i > j $ and non-divisible by $f_{\tau(j)}$.
Furthermore, $\nu_L(\beta_i) = \nu_L(\alpha_i)$ for all $0\neq \beta_i \in  \erz{\alpha_i}_{\ef_p[H] }$. 
We then have 
\begin{align*}
\disc(L_\alpha/L ) & \overset{\eqref{ElementaryAbelian-Disc}}{=} \summ_{ 0\neq \beta \in \erz{\alpha}_{\ef_p[H]}} \ASF_L(\beta) 
=  \summ_{ j=1}^{ \cardI } \delta_j(\tau,I) \cdot \ASF_L(\alpha_{\tau(j)}) \\ 
&\overset{\eqref{number-MJ}}{=}  \summ_{j=1}^{ \cardI }  p^{\ell(\tau(1))} \cdots p^{\ell(\tau(j-1))} \cdot \klmm{ p^{\ell(\tau(j))}  -1} \ASF_L(\alpha_{\tau(j)}) .
\end{align*} 
For  part (c), use the discriminant tower formula, the discriminant formula for tamely ramified extensions and observe  $p^{\ell(\tau(1))} \cdots p^{\ell(\tau(j-1))} \klmm{ p^{\ell(\tau(j))} -1 } = p^{R_{j}} -  p^{R_{j-1} }  $.
\end{proof}
 If $I=\{ i\}$ is a singleton, we get $\alpha = \alpha_i$. Let us define $M=L \klmm{ \wp\inv(\alpha) }$. Then we obtain by Proposition~\ref{Disc-Thm}(b)
\[ 
\disc(L_\alpha/L ) = (p^{\ell(I)} -1)\ASF_L(\alpha) = \frac{p^{\ell(I)}-1}{p-1}\disc(M/L).
\] 
\begin{rem}\label{disc_single_I}
If $I=\{i\}$ then we get
\[\disc(L_\alpha/F) = \frac{p^{\ell(I)}-1}{p-1} \disc(M/F) -f(e-1)\frac{p^{\ell(I)}-p}{p-1}.
\]
\end{rem}
\begin{proof}
    $\disc(L_\alpha/F) = f\disc(L_\alpha/L)+ p^{\ell(I)} \disc(L/F)$
    \[= \frac{p^{\ell(I)}-1}{p-1}f\disc(M/L) +p^{\ell(I)} \disc(L/F)
    \]
    Using $\disc(M/F) = f\disc(M/L) +p \disc(L/F)$ we get
    \[\disc(L_\alpha/F) = \frac{p^{\ell(I)}-1}{p-1}\left(\disc(M/F)-p\disc(L/F)\right)+ p^{\ell(I)} \disc(L/F)\]
\[    =\frac{p^{\ell(I)}-1}{p-1}\disc(M/F)+ \frac{(p-1)p^{\ell(I)}-p(p^{\ell(I)}-1)}{p-1}\disc(L/F)
    \]
\[    =\frac{p^{\ell(I)}-1}{p-1}\disc(M/F)- \frac{p^{\ell(I)}-p}{p-1}\disc(L/F).
    \]
    The assertion follows by using $\disc(L/F)=f(e-1)$.
\end{proof}
%

\section{\texorpdfstring{Counting Business over $pd$ Points}{Counting Business over pd Points}}

Next, we need to determine  the number of $\alpha \in Y_I$ with $Y_I $ as in \eqref{Covidl} up to some valuation bound to solve the asymptotics problem over $pd$ points.   Let $I\subseteq \{ 1 ,\ldots, r \}$ and $x > 0 $.
We compute the numbers 
\begin{align*}
N_I(L,x)	&:= \# \left\{ \alpha  \in Y_I ~:~  |\nu_L(\alpha)| \leq x, \ \alpha_i \neq 0  \text{ for } i \in I  \right\}, \\
N_i(L,x) &:= N_{ \{ i\}}(L,x),
\end{align*}
which yield a nice formula for the arithmetic progression $(ep, 2 ep,\ldots )$ due to the $ep$-periodicity of the underlying module $W_z$, see~\eqref{Decomposition-Wn}.


\begin{lem}\label{Number-NI}
	Let $I\subseteq \{ 1,\ldots, r\}$, $x > 0$ and $n \in \en$. 
	\begin{enumb}
\item We have	\[ 
		N_I(L, x ) = \prod_{i\in I} N_{i}(L, x) .
		\] 
\item 
Let $\kappa(I):=\begin{cases}
	1, & 1\in I \\
	0, & 1\notin I
\end{cases}$ and  $\kappa(i) :=  \kappa(\{ i \})$. 
Then we have 
\begin{align*}	N_{  i }(L, n \cdot ep) & = 
p^{\kappa(i) } 	q^{ n\cdot \ell(i)(p-1) } -1 \quad \text{and}\quad
N_I(L, n \cdot ep)  \sim p^{\kappa(I) } q^{n \cdot \ell(I) (p-1)}.
	\end{align*}
 \end{enumb}
\end{lem}
\begin{proof}
Part (a) is clear by~\eqref{Decomposition-RL} and $\sigma\klmm{ \ef_{q^f} \pi_L^k} = \ef_{q^f} \pi_L^k$.

Concerning part (b), we use 
  \[ Y_{f_i}(R_L) \overset{\eqref{Yi-dec} }{\cong} Y_{f_i} \klmm{ \ef_p \omega_0} \oplus \bigoplus_{z < 0}
Y_{f_i}\klmm{W_z}\] with  $ |Y_{f_i}(W_z)| \overset{\eqref{Anzahl-YI}  }{=} q^{\ell(i)(p-1)}$ and where $ Y_{f_i}(\ef_p \omega_0) = \begin{cases}\ef_p \omega_0, & i= 1 \\ 0, & i \neq 1\end{cases}$ due to the fact that $\omega_0$ is an eigenvector of $\sigma$ of eigenvalue $1$, hence in the kernel of $f_1(\sigma)$.

In the decomposition, we only need to consider 
\begin{equation}\label{eq:W}
	W(n):= \ef_p \omega_0 \oplus W_{-1}\oplus \ldots \oplus W_{-n}
\end{equation} 
by the valuation constraint in $N_i(L,n\cdot ep)$ and obtain 
 \begin{equation} \label{eq-Ninep} N_{ i}(L,n\cdot ep) =  |Y_{f_i}(W(n)) \setminus \{ 0\}| = 
 p^{\kappa(i)} \cdot q^{ n(p-1)\ell(i) } -1 
 \end{equation}
by removing the zero element.

For a subset $I$, we get using (a), \eqref{eq-Ninep} and the inclusion-exclusion principle
\begin{align*}
	N_I(L,n\cdot ep) 
	 &\overset{(a)}{\underset{\eqref{eq-Ninep}}{=} }
	 \prod_{i \in I  } \klmm{ p^{\kappa(i)}q^{n\cdot \ell(i)(p-1)  } -1 }\\ &= \sum_{J\subset I}(-1)^{\cardI - \cardJ} \prod_{j\in J }p^{\kappa(j)}
	\cdot q^{n \cdot \sum\limits_{j\in J} \ell(j) (p-1) }  \\
	&= \sum_{J\subset I}(-1)^{\cardI - \cardJ} 
	\cdot q^{n\cdot \ell(J)(p-1)} p^{\kappa(I)} \sim p^{\kappa(I) } q^{n \cdot \ell(I) (p-1)}. 
\end{align*}
Clearly, $ q^{n\ell(I)(p-1)} \prod\limits_{i\in I}p^{\kappa(i)} $ is asymptotically dominant, and 
 $\prod\limits_{i\in I } p^{\kappa(i)} = p^{\kappa(I)}$. 
\end{proof}

Next, we give a description for $N_{i }(L,x)$ for arbitrary $x >0$. 
For every $i \in \{ 1,\ldots, r\}$, there is an integer $0\leq \rho_L(i) \leq e-1$ such that 
\begin{equation}\label{General-Ni(L,x)-Dec}
Y_i \cong  \ef_p \cdot \kappa(i) \omega_0 \pi_L^0 \oplus  \bigoplus_{ \substack{ k<0 \\ p \nmid (ek+ \rho_L(i))} } \ef_p^{\ell(i) [\ef_q:\ef_p]} \pi_L^{ek + \rho_L(i)},
\end{equation} where we use Proposition~\ref{thm:Fq-L-Structure} and the more detailed description of $W_z$. 
We get using Lemma~\ref{Number-NI}, Theorem~\ref{thm:Fq-L-Structure}(b) and \eqref{Decomposition-Wn}:
\begin{lem}
	\label{La:Ni-Wachstum}
 For fixed $L$ and $i \in \{ 1,\ldots, r\}$ there are unique $0 \leq \rho_L(i) \leq e-1$ and $\tilde \rho_L(i) := ez_i + \rho_L(i)$ for $0\leq z_i \leq p-1$ with $0 \leq \tilde \rho_L(i) \leq ep-1$ such that
 \[\tilde \rho_L(i)\equiv \rho_L(i) \bmod e\text{ and }\tilde \rho_L(i) \equiv 0 \bmod p.\]
\begin{enumb} 
	\item  For $x = nep + y$ with $n\in \mathbb N$ and $0\leq y < ep$ we have
	\[ 
	N_{ i }(L, x) = N_i(L,n\cdot ep+y)=
		p^{\kappa(i)} \cdot q^{ n(p-1)\ell(i)}q^{ \gamma_i(L,y)\ell(i) } -1,
	\]
	where
	\begin{equation*}
	\gamma_i(L,y) = \ptrunc{y- \rho_L(i)}{e} - \ptrunc{y-\tilde \rho_L(i)}{ep} \in \{ 0,1,\ldots, p-1\}.
		\end{equation*} 
	\item Writing $y = \lambda \cdot e + \rho_L(i)$ for $0\leq \lambda \leq p-1 $, we have 
	\begin{equation*} 
			\gamma_i \left(L,\lambda e + \rho_L(i) \right) = \begin{cases} 
				\lambda -1, & \lambda \geq z_i \\
				\lambda, & \lambda < z_i.
				\end{cases} 
	\end{equation*}
\item For $y\in\mathbb{N}$ with $0\leq y < ep$ we have
\begin{equation}\label{anzahlvi}
\widetilde N_i( L; nep+y) := N_{  i  }( L,  nep+y) -N_{  i  }( L,  nep+y -1)
\end{equation}
\begin{equation} \label{eq-Ntilde-Anzahl}
	\overset{\eqref{diff}}{=} \begin{cases}
		p^{\kappa(i)} \cdot q^{ n(p-1)\ell(i)} (q^{\ell(i)}-1)q^{\ell(i)(\gamma_i(L,y-1)) }  & \text{if }y \equiv \rho_L(i) \bmod e, p\nmid y\\
		0 & \text{otherwise}
	\end{cases}
\end{equation}
\end{enumb}
\end{lem}
\begin{proof}
Using \eqref{Anzahl-YI} and \eqref{eq-Ninep} we get $n$ full blocks and therefore using \eqref{eq:W} we get
\[|Y_{f_i}(W(n))| = 	p^{\kappa(i)} \cdot q^{ n(p-1)\ell(i)}.\]
We consider~\eqref{General-Ni(L,x)-Dec} to analyze the partial block. Note that  	
the $\ef_p$-dimension of $Y_i$ increases by $\ell(i)$ exactly at the integers which are congruent to 
$\rho_L(i)$ modulo $e$ by \eqref{General-Ni(L,x)-Dec} and which are not divisible by $p$ by Artin-Schreier 
theory. This is exactly described by $\gamma_i(L,y)$, and the identity in (b) follows immediately.
For (c) note that
\begin{equation}\label{diff}
	\gamma_i(L,y)-\gamma_i(L,y-1)=\begin{cases}
		1 &  y \equiv \rho_L(j) \bmod e \text{ and }p\nmid y\\
		0 & \text{otherwise}
	\end{cases}
\end{equation}
and (c) follows.
\end{proof}

\begin{lem}\label{lem3}
	For $x = nep + y$ with $n\in \mathbb N$ and $0\leq y < ep$ we define
	\[ \gamma_I(L,y):=\sum_{i\in I} \gamma_i(L,y)\ell(i) \mbox{ and }
	\delta_I(L,y):=\gamma_I(L,y) - \frac{y(p-1)\ell(I)}{pe}
	\] 
	$pe$-periodic functions. 
	Then we get
\begin{enumb}
\item	\(
	N_I(L,n\cdot ep+y)\;\sim_{n\rightarrow \infty}\;
	p^{\kappa(I)} \cdot q^{ n(p-1)\ell(I)} q^{ \gamma_I(L,y)}.\)
\item \(N_I(L,x)\; \sim_{x\rightarrow\infty}\;p^{\kappa(I)} \cdot q^{ x(p-1)\ell(I)/(pe)} q^{ \delta_I(L,y)}. \)
\end{enumb}
\end{lem}
\begin{proof}
	Using Lemma \ref{Number-NI} (a) and Lemma \ref{La:Ni-Wachstum} (a) we get:
	\[N_I(L, x) = N_I(L,n\cdot ep+y) = \prod_{i\in I} \left(	p^{\kappa(i)} \cdot q^{ n(p-1)\ell(i)}q^{ \gamma_i(L,y)\ell(i) } -1\right)\] 
\[	\sim \prod_{i\in I}	p^{\kappa(i)} \cdot q^{ n(p-1)\ell(i)}q^{ \gamma_i(L,y)\ell(i) } = p^{\kappa(I)} \cdot q^{ n(p-1)\ell(I)}
 q^{ \gamma_I(L,y) }.\]
The last formula we get using the substitution $n=\frac{x-y}{ep}$.
\end{proof}


In the following, we consider the constant
\begin{equation}\label{Psi}
\Psi(I) :=  \prodd_{i\in I} (p^{\ell(i)} - 1 ) . 
\end{equation}
which is the number of cyclic module generators of $ M_I$, where $M_I  = \ef_p[X]/( \prod\limits_{i \in I} f_i).$

We now consider the nicer formula from Lemma~\ref{Number-NI} and consequences on the desired discriminant counting problem. For a fixed $C_d$-extension $L/F$ we define $Z_{L,pd}(F,G_p(d,I); N) :=$
\[
 \# \{ M/F ~:~ \Gal(M/F) \cong G_p(d,I), \disc(M/F) \leq N, \  L\leq M \}.
\]
\begin{thm}\label{thm-deltan} 	Let $\emptyset \neq I\subseteq \left\{ 1,\ldots, r \right\}$. Then there exists a $pd(p-1)$-periodic function $\beta_{I,L}(x)$ such that for $x\rightarrow \infty$ we have
\[ Z_{L,pd}(F,G_p(d,I);x) \sim \beta_{I,L}(x) q^{x\frac{\ell(I)}{pd}}.\]
\end{thm}
\begin{proof}
Every $G_p(d,I)$-extension containing $L$ is given as $  L \klmm{ \wp\inv(\alpha) } $ for some $\alpha \in Y_I$ by 
Propositions~\ref{Gal-Disc} and \ref{RL-Iso}. Note that every cyclic generator of $\ef_p[H]\cdot \alpha$ leads to an $F$-isomorphic field to $ L \klmm{ \wp\inv(\alpha) }$, and there are $\Psi(I)$ many cyclic generators for this module,
 For $\nu_L(\alpha)<0$, we have using \eqref{pd-disc}
\[\disc( L \klmm{ \wp\inv(\alpha) }/F) = pd-f + (p-1)f |\nu_L(\alpha)|.\]
 Then we   have using Lemma \ref{lem3} 
  \[ Z_{L,pd}(F,G_p(d,I);pd-f +(p-1)f z) = \frac{ N_I(L; z) }{\Psi(I)}\]\[
   \sim_{z\to \infty} \frac{ p^{\kappa(I)}q^{ \delta_I(L,y)}}{\Psi(I)} q^{z \cdot \ell(I)(p-1)/(pe) }.\] 
 Now we have \[x=pd-f+(p-1)fz \Leftrightarrow  z = \frac{x-pd+f}{(p-1)f}. \]Using this substitution the $ep$-periodic function is replaced by a $epf(p-1)= dp(p-1)$-periodic function $\beta_{I,L}(x)$, which also takes care about the constants $p^{\kappa(I)}$ and $\Psi(I)$, and the pre-period. The final step follows by the identity $d=ef$.
\end{proof}

Note that the module structure is $pe$-periodic, while the discriminant counting function is basically $(p-1)pd$-periodic, due to the powering with the inertia degree of the relative discriminant in the discriminant tower formula. Let us define
\begin{equation}\label{def_counting}
Z_{pd}(F,G_p(d,I); x) :=
 \# \{ M/F ~:~ \Gal(M/F) = G_p(d,I), \disc(M/F) \leq x\}.
\end{equation}
for counting extensions of degree $pd$.
\begin{thm}\label{pd-points}
There exists a $(p-1)pd$-periodic function $\beta_I(x)$ such that
considered as permutation group over $pd$ points, we get for $x\rightarrow \infty $:
\[  Z_{pd}\klmm{ F,G_p(d,I);x } \sim \beta_I(x) q^{x \frac{\ell(I)}{pd}}. 
\]
\end{thm}
\begin{proof}
 There are finally many cyclic $C_d$-extensions $L/F$. Using a summation over these $L$ we get:
\[
Z_{pd}(F,G_p(d,I);x) = \sum_{ \substack{ L/F \\ \Gal(L/F) \cong C_d } }  Z_L(F,G_p(d,I);x )\]
\[ \sim \sum_{ \substack{ L/F \\ \Gal(L/F) \cong C_d } } \beta_{I,L}(x)  q^{x \frac{\ell(I)}{pd}} \sim \beta_I(x) q^{x\frac{\ell(I)}{pd}}.\qedhere
\]
\end{proof}
Note that different $I, \tilde I$ may lead to isomorphic groups $G_p(d,I)\cong G_p(d,\tilde I)$. Certainly, there are only finitely many possibilities. 

\section{\texorpdfstring{Counting Business over $p^{\ell(I)}d$ Points}{Counting Business over p\textasciicircum ell(I)d Points}}

In this chapter we consider the discriminant counting for the $G_p(d,I)$ considered as permutation groups on $p^{\ell(I)}d$ points, i.e. we solve the discriminant counting problem for the splitting fields. 

\begin{thm} \label{thm:Asymptotik-Zerf} 
	There exists a $(p^{\ell(I)}-1)pd$-periodic function $\tilde\beta_I(x)$ such that
	considered as permutation group over $p^{\ell(I)} d$ points, we get for $x\rightarrow \infty $:
	\[  Z_{p^{\ell(I)}d}\klmm{ F,G_p(d,I)};x ) \sim \tilde \beta_I(x)  q^{ x\frac{p-1}{pd }\frac{ \ell(I) }{p^{\ell(I) } -1} }. \]
\end{thm}
As argued in the previous chapter, there exist only finitely many $C_d$-extensions $L/F$ and this theorem will immediately follow from the next theorem. 
\begin{thm} \label{thm:Asymptotik-ZerfL}
	Let $L/F$ be a $C_d$-extension and \(Z_{L,p^{\ell(I)}d}\klmm{ F,G_p(d,I)};x ) :=\)
	\[ \# \{ M/F ~:~ \Gal(M/F) \cong G_p(d,I), \ \disc(M/F) \leq x, \  L\subseteq M \}
	\]
	with $G_p(d,I)$ considered as permutation group over $p^{\ell(I)} d$ points. Then
	there 	exists a $(p^{\ell(I)}-1)pd$-periodic function $\tilde\beta_{I,L}(x)$ such that
	for $x\rightarrow \infty$ we get 
	\[  Z_{L,  p^{\ell(I)}d}\klmm{ F,G_p(d,I)};x ) \sim \tilde\beta_{I,L}(x) q^{x \frac{p-1}{pd }\frac{ \ell(I) }{p^{\ell(I) } -1} }. \]
\end{thm}

\smallskip
\begin{rem}
In particular, all extensions with a fixed intermediate $C_d$-extension $L$ have a positive density in Theorem \ref{thm:Asymptotik-Zerf}. 
\end{rem}
The rest of this section is devoted to prove this theorem. One special case is easy.

\begin{lem}
	Theorem \ref{thm:Asymptotik-ZerfL} is true for the case $\#I=1$, i.e. for $I=\{i\}$.
\end{lem}
\begin{proof}
	we choose an $\alpha\in Y_I$ as in the proof of Theorem \ref{thm-deltan}. Denote by $L_\alpha$ the splitting field over $F$ and let $M=L \klmm{ \wp\inv(\alpha) }$. Using Remark \ref{disc_single_I} with $c_{L,I}:=f(e-1)\frac{p^{\ell(I)}-p}{p-1}$ we get:
\[	\disc(L_\alpha/F) = \frac{p^{\ell(I)}-1}{p-1} \disc(M/F) -c_{L,I}.\]
Using this and Theorem \ref{pd-points} we find a $(p^{\ell(I)}-1)pd$ periodic function $\tilde\beta_I$ such that
	\[  Z_{p^{\ell(I)}d}\klmm{ F,G_p(d,I)};x ) \]
	\[=   Z_{pd}\klmm{ F,G_p(d,I)}; (x+c_{L,I}) / \frac{p^{\ell(I)}-1}{p-1} )
\sim \tilde\beta_I(x) q^{x \frac{p-1}{p^{\ell(I)-1}}\frac{\ell(I)}{pd}}. \qedhere\]
\end{proof}

\subsection{\texorpdfstring{The case $J=\# I>1$.}{The case J=|I|>1}}


We will fix a $C_d$-extension $L/F$ as before. 

We describe an extension $L_\alpha$ via
\[\alpha=\sum_{i\in I}\alpha_i\mbox{ with }\alpha_i\ne 0\mbox{ for }i\in I.\]
For computing the discriminant we have to sort the $\alpha_i$ via some bijection $\tau: \{1,\ldots,J\} \rightarrow I$. 
\[ 
	U_{I,\tau}:=\{\alpha = 
	\sum_{i\in I}\alpha_i\in Y_I \mid \alpha_i\neq 0, \ |\nu_L(\alpha_{\tau(1)})| \leq \ldots \leq |\nu_L(\alpha_{\tau(J)})|  \}.
\]
This $\tau$ is not unique if different $\alpha_i$ have the same valuation. The number of all such $\tau$ is:
\[\Lambda(b_1,\ldots,b_J):=|\{ \pi \in \Sym(J) \mid b_{\pi(1)} \leq b_{\pi(2)} \leq \ldots \leq b_{\pi(J)}\}|.  \]

Using formula \eqref{eq:disc-Ltheta-alpha-Spl2} the discriminant exponents are only depending on the valuations of the $\alpha_i$.

\begin{dfn} For  
 $c_1,\ldots, c_J \in \mathbb N$ and $\alpha \in L$, we define
\begin{align*}
	\disc_{L,I,\tau}(c_1,\ldots, c_{J } ) & :=  p^{\ell(I)} f(e-1) +  \sum_{j=1}^{ J } f  \left( p^{R_j} - p^{R_{j-1}} \right) (c_j+1),  
	\\
	\disc_{L,I,\tau}(\alpha) & := \disc_{L,I,\tau}(|\nu_L(\alpha_{\tau(1)})|, \ldots, |\nu_L(\alpha_{\tau(J)})|).	
\end{align*}
\end{dfn}

\begin{lem}\label{lem:Disc-Zeug}
	Let $\alpha \in U_{I,\tau}$ and $c_j := \ASF_L(\alpha_{\tau(j)})$. 
	Set $r_j := \ell(\tau(j))$ and $R_j := \sum\limits_{k=1}^j r_k$. Then we have 
	 \[ 
\disc(L_\alpha/F)	= \disc_{L,I,\tau}(\alpha) := p^{\ell(I)} f(e-1) +  \sum_{j=1}^{\cardI } f  \left( p^{R_j} - p^{R_{j-1}} \right) c_j.
	  \]
\end{lem}
\begin{proof}
	As $|\nu_L(\alpha_{\tau(1)})| \leq \ldots \leq |\nu_L(\alpha_{\tau(J)})|$, this follows by formula~\eqref{eq:disc-Ltheta-alpha-Spl2}.
\end{proof}

\smallskip
\begin{lem}\label{lem:tau}
Write  $J:= \cardI $. The Dirichlet series
\[
\Phi_L(s) := \sum_{ \substack{ M/L \\ \Gal(M/F) \cong  G_p(d,I)  }} q^{-\disc(M/F)s}
\] has a decomposition
\begin{align*}
	\Phi_L(s) &= \frac{1}{\Psi(I)} \sum_{\tau \colon \{ 1,\ldots, J \} \to I \text{ bij.}} \Phi_\tau(s), \text{ where }\\ 
	\Phi_\tau(s) &:= \sum_{ \alpha \in U_{I,\tau}}  \frac{q^{-s\disc_{L,I,\tau}\left(\alpha\right)  }}{ \Lambda\left( |\nu_L(\alpha_{\tau(1)})|, \ldots, |\nu_L(\alpha_{\tau(J)})|\right) }\\
 &= 	\sum_{ \substack{(n_1,\ldots, n_J) \in \mathbb{N}^J  \\ n_1 \leq n_2 \leq \ldots \leq n_J } }\frac{q^{-s\disc_{L,I,\tau}\left(n_1,\ldots, n_{J} \right)  }}{ \Lambda\left( n_1,\ldots,n_J\right)}  
\prod_{j=1}^J \widetilde N_{\tau(j)} \klmm{L, n_j }
\end{align*}
and $\tilde N_j$ resp. $\Psi(I)$ are defined in \eqref{anzahlvi} resp. \eqref{Psi}.

\end{lem}
\begin{proof}
	For each $ G_p(d,I)$-extension $M/F$ containing $L$ there exist precisely $\Psi(I)$ elements $\alpha \in U_I$ such that 
	$M = L_\alpha$ (see~\eqref{Psi}). The existence is clear, and $L_\alpha = L_\beta$ holds if and only if $\alpha$ and $\beta $ generate the same 
	$\ef_p[C_d]$-module for which there are precisely $\Psi(I)$ possibilities. 
	
	Finally, each element $\alpha = \sum\limits_{i\in I} \alpha_i$ with $\alpha_i\neq 0$ is contained in 
	at least one of the $U_{I,\tau}$ by sorting  $\ASF_L(\alpha_i)$, thus each field is considered in the 
	sum over $\tau$. For each $\alpha $ exist exactly $\Lambda\left( 
	|\nu_L(\alpha_{1})|, \ldots, |\nu_L(\alpha_{J})|\right) $ many permutations $\tau$ which preserve the ordering by definition.
\end{proof}

\subsection{General Case}

A composition of  $N\in \mathbb N$ is a tuple $ \omega =(\omega_1,\ldots, \omega_{\lambda(\omega)})$ with $\omega_i \in \mathbb N^+$ so that $\displaystyle \sum\limits_{i=1}^{\lambda(\omega)} \omega_i = N$.  We write $\Omega_J:= \{ \omega \mid \omega \text{ composition of }J\}$ 

For an ordered $n$-tuple $c:=(c_1,\ldots, c_n) \in \mathbb N^n$, we associate a composition 
$\comp(c) := (\omega_1,\ldots, \omega_{\lambda(\omega)})$, where $ \left(c_1,\ldots, c_{n} \right) =\! \left( \underbrace{ c_1,\ldots , c_1}_{\omega_1 \text{ times}}, 
\underbrace{ c_{\omega_1+1} ,\ldots, c_{\omega_1+1} }_{\omega_2 \text{ times} }, \ldots  \right).$ 

If a chain $(c_1,\ldots, c_n)$ has the composition $\omega $, we clearly have $  \Lambda(c_1,\ldots,c_n) = \prod\limits_{i=1}^{\lambda_\omega} \omega_i ! =:\Lambda(\omega) $
as the product of the faculties of the components of its composition. 

\smallskip  For $j=1,\ldots, J = \cardI$ and  $i=1,\ldots, p-1$, we set 
\begin{equation*}
\tilde l_j(i) := \begin{cases} (i-1) e + \rho(j) ,  & i-1 < \ptrunc{z_j - \rho(j)}{e}, \\
		i e + \rho(j) , & i-1 \geq \ptrunc{z_j - \rho(j)}{e} 
\end{cases} \end{equation*}
running through all residues $\mod{e}$ congruent to $\rho(j)$ and not divisible by $p$. 

These numbers serve as a link between the dimension weights
and the appearing conductor exponents respectively discriminant exponents. More precisely, the number $k_j ep + \tilde l_j(i)$ corresponds to dimension $ (p-1)k_j + i$ linked to the discriminant exponent 
$  f\left(p^{R_j - p^{R_{j-1}}}\right)  \left( pe k_j + \tilde l_j(i) +1\right)$.
We further define
\begin{equation*} 
	\mathcal C_I(\tau) := \left\{ \left( pe k_1 + \tilde l_1(v_1),\ldots , pe k_{J} + \tilde l_{J} (v_{J})  \right) : k_1\leq k_2\ldots \leq k_{J}; \ 1\leq v_j \leq p-1 \right\}
\end{equation*}  as the set of $(I,\tau)$-chains. 
We call a $(I,\tau)$-chain $ \mu = (\mu_1,\ldots , \mu_J) \in  C_I(\tau)$ \emph{feasible} if $\mu_1\leq \mu_2\ldots \leq \mu_J$ holds. 
Note that these $(I,\tau)$-chains are not necessarily ordered and thus, do not correspond to 
field extensions in this case.  We consider non-feasible chains as well to simplify the argumentation.

For a chain 
\begin{equation}\label{eq:chain}
 \mu = \left( pe k_1 + \tilde l_1(v_1),\ldots , pe k_{J} + \tilde l_{J} (v_{J})  \right) \in  C_I(\tau),
\end{equation}
 we write  
\[k(\mu) := (k_1,\ldots, k_J)\text{ and }v(\mu) = \left( v_1, \ldots , v_J  \right).\]  
We associate the composition $\comp(k(\mu))$ of $J$ to this $\mu$, for technical reasons, 
although it might appear slightly unnatural.

For the proof of
Theorem~\ref{thm:Asymptotik-ZerfL}, we decompose the function $\Phi_\tau$ given in Lemma \ref{lem:tau} into subsums corresponding to the compositions $\operatorname{comp}\left(k(\mu) \right) \in \Omega_J$. For this we define
\[ 
U_{I,\tau,\omega}:=\{\alpha = 
\sum_{i\in I}\alpha_i\in U_{I,\tau} \mid  
\comp\left(k(|\nu_L(\alpha_{\tau(1)})|, \ldots , |\nu_L(\alpha_{\tau(J)})|)\right)=\omega\}.
\]
This certainly gives
\[U_{I,\tau}=\mathop{\dot{\bigcup}}\limits_{\omega\in \Omega_J}  U_{I,\tau,\omega}.\]
\begin{lem} \label{lem:tau2}
	\[ 
\Phi_\tau(s) = \sum_{ \omega \in \Omega_{J } } \Phi_{\tau,\omega}(s), 
\] 
where
\[
\Phi_{\tau,\omega}(s) = 	\sum_{ \substack{\mu =(n_1,\ldots, n_J) \in \mathcal C_I(\tau) \\ n_1 \leq n_2 \leq \ldots \leq n_J  \\ \operatorname{comp}\left(k(\mu)\right) = \omega } }\frac{q^{-s\disc_{L,I,\tau}\left(n_1,\ldots, n_{J} \right)  }}{ \Lambda\left( n_1,\ldots,n_J\right)}  
\prod_{j=1}^J \widetilde N_{\tau(j)} \klmm{L, n_j }
\]
and $\tilde N_j$ are defined in \eqref{anzahlvi}.
\end{lem}
\begin{proof}
	We start with Lemma \ref{lem:tau}.
	The condition $n_1 \leq n_2 \leq \ldots \leq n_J$ guarantees that the $(I,\tau)$-chains are feasible. The product computes the number of fields which yield the same tuple in $C_{I}(\tau)$.
\end{proof}

The $\Lambda$-value of $\mu$ can be computed from $k(\mu)$ and $v(\mu)$.
\begin{dfn}
	For a chain $ \mu$ given in  \eqref{eq:chain} we define
	\[\Lambda(k(\mu),v(\mu)) := \Lambda(\mu_1,\ldots,\mu_J). \]
\end{dfn}

We will focus mainly on $\omega:=\comp(k(\mu))$ now. For $1\leq i \leq \lambda(\omega)$ we set
\begin{equation*} 
	A_i(\omega) := \sum_{j=1}^i  \omega_j .
\end{equation*}
For $\omega = \comp(k(\mu))$, we have per construction, 
\begin{equation}\label{comp-Ajs} (k_1,\ldots, k_J) = ( \omega_1, \ldots , \underbrace{\omega_1}_{A_1(\omega)} , \underbrace{\omega_2}_{A_1(\omega) +1}, \ldots, \underbrace{\omega_2}_{A_2(\omega)}, \ldots, \underbrace{\omega_{\lambda(\omega)} }_{ A_{\lambda(\omega)-1}+1}, \ldots, \underbrace{\omega_{\lambda(\omega)}}_{A_{\lambda(\omega)}}  ) .\end{equation}
The following remark gives criteria for feasible $(I,\tau)$-chains.  

\smallskip
\begin{rem} \label{feasible}
	Let  $\omega \in \Omega_J$ be a composition and $v \in \{1,\ldots, p-1\}^J$ and define \[ 
	\widetilde c\left( \omega, v \right) := \begin{cases} 1, & \tilde l_{ A_{i-1}(\omega) +j}( v_{A_{i-1}(\omega) +j } ) \leq \tilde l_{ A_{i-1}(\omega) +j+1}( v_{ A_{i-1}(\omega) +j+1 })\\ &  \text{ for  }  i=1,\ldots, \lambda(\omega),\text{ and } j=1\ldots, \omega_i-1. \\
		0 , & \text{else.} 
	\end{cases}
	\] 
	\begin{enumb}
	\item Let 
	$ \mu = (pek_1 + \tilde l_1( v_1),\ldots , pek_J + \tilde l_J(v_J)) \in 	\mathcal C_I(\tau)$ with $\comp(k (\mu)) = \omega$. Then, $\mu$ is feasible if and only if $\widetilde c\left( \omega, v \right) =1$. 
	\item For $\omega =(1,\ldots, 1)$, we always have $\widetilde c\left( \omega, v \right) =1$. 
\end{enumb}
\begin{proof}
Note that $\widetilde l_j(i_j) \leq  e(p-1) + e-1 < ep$ and hence  
\[  pek_j + \widetilde l_j(i_j) \le   pek_{j+1} + \widetilde l_{j+1} (i_{j+1})\] \[\iff
	k_j < k_{j+1} \text{ or } k_j = k_{j+1} \text{ and } \widetilde l_{j} (i_{j}) \le  \widetilde l_{j+1} (i_{j+1}).
\]
Thus, the ordering depends only on the values of $\widetilde l_j (i_j)$ where the $k$-values coincide.
(a) follows and by this observation (b) is immediate as $\mu_k=(1, \ldots, 1)$ implies
$k_1<k_2 <\ldots < k_J$.
\end{proof}
\end{rem}


The number of components $\alpha_i$ is given by the formula for $\widetilde N_i(L,x)$ which was calculated in \eqref{eq-Ntilde-Anzahl} in Lemma~\ref{La:Ni-Wachstum}.

\smallskip 
\begin{thm}\label{Phi-Tau-Zerlegung}
	The counting function $\Phi_\tau(s)$ is rational and has a decomposition 
	\[ 
	\Phi_\tau(s) =\sum_{ \omega \in \Omega_{J } } \Phi_{\tau,\omega} =  \sum_{ \omega \in \Omega_{J } } \Psi_\omega(s) \sum_{ v \in \{ 1,\ldots, p-1\}^{J}  } c(\omega, v) h_v(s),
	\]
	with constants $c(\omega,v)$ defined in \eqref{c-omega-v},  holomorphic functions 
\[h_v(s) =  q^{-p^{\ell(I)} f(e-1) s} \prod\limits_{i=1}^J  	q^{ -s  \klmm{ p^{R_i }- p^{ R_{i-1}} } ( \tilde l_{i}(v_i) +1 ) }
	q^{\ell(i) v_i} \] and meromorphic functions  
\begin{equation} \label{eq:Psi-omega-h} \Psi_{\omega}(s) = 
\sum_{ \left( k_1< \ldots < k_{\lambda(\omega)} \right) } 
\prod_{i=1}^{\lambda(\omega)} q^{ k_i \sum\limits_{t=A_{i-1}(\omega)+1}^{A_i (\omega)} (p-1) \klmm{ r_t -spd \klmm{  p^{R_t} - p^{R_{t-1}}  }	} }.
\end{equation} 
\end{thm}
\begin{proof} 
Using Lemma \ref{lem:tau} we have
\[	\Phi_\tau(s)  = 
	\sum_{ \substack{(n_1,\ldots, n_J) \in \mathcal C_I(\tau) \\ n_1 \leq n_2 \leq \ldots \leq n_J } }\frac{q^{-s\disc_{L,I,\tau}\left(n_1,\ldots, n_{J} \right)  }}{ \Lambda\left( n_1,\ldots,n_J\right)}  
	\prod_{j=1}^J \widetilde N_{\tau(j)} \klmm{L, n_j }.\]

The functions  $\widetilde N_{\tau(j)}$ are as defined in \eqref{anzahlvi}, where 
$\widetilde N_{\tau(j)}(n_j) \neq 0 $ if and only if 
$ n_j= pek_j + \tilde l_j(v_j)$ for some $0\leq v_j\leq p-1$ and $k_j \in \mathbb N$. 
Thus, the sum is ranging over all chains in $\mathcal C_I(\tau)$ which are increasingly ordered, namely the feasible ones. Thus, we get:
\begin{align*} 
\Phi_\tau(s) 
\overset{}{=} &  \sum_{\omega \in \Omega_J} \sum_{v\in \{1,\ldots, p-1\}^J} 
\sum_{ \substack{  \mu 
		\in \mathcal C_I(\tau)  \text{ feasible}
\\ \omega_{k(\mu)} = \omega \\ v_\mu = v } } \frac{q^{-s\disc_{L,I,\tau}\left( pek_1 + \tilde l_1(v_1),\ldots , pek_J + \tilde l_J(v_J) \right)  }}{ \Lambda\left( \omega, v \right)}\\
&\prod_{j=1}^J
 \widetilde N_{\tau(j)} \klmm{L, pek_j + \tilde l_j(v_j) } 
 \\
\overset{ \eqref{eq-Ntilde-Anzahl} }{\underset{\text{La.~}\ref{lem:Disc-Zeug} }{=} }&  \sum_{\omega \in \Omega_J} \sum_{v\in \{1,\ldots, p-1\}^J} 
 q^{-p^{\ell(I)} f(e-1) s}\!\!\!\!\!\!
\sum_{ \substack{  \mu  \in \mathcal C_I(\tau)   \text{ feasible}
		\\ \omega_{k(\mu)} = \omega \\ v_\mu = v } }\!\!\!\!\!\!
\frac{
\prod\limits_{j=1}^J q^{ -s f\klmm{ p^{R_j }- p^{ R_{j-1}} } \klmm{ pe k_j + \tilde l_j(v_j) +1} }
}{{ \Lambda\left(  \omega,v \right)}}\\ &  \quad
\prod_{j=1}^J p^{\kappa(j)}\left( q^{r_j } -1 \right) q^{(p-1) k_j r_j } q^{(v_j -1) r_j } 
 \\
\overset{\text{Rem.~}\ref{feasible}}{=}& \sum_{\omega \in \Omega_J} \sum_{v\in \{1,\ldots, p-1\}^J} 
 \widetilde{c}(\omega,v)\frac{\prod\limits_{j=1}^J p^{\kappa(j)}\left( q^{r_j } -1 \right) }{{ \Lambda\left(  \omega, v \right) }} \\ 
 & \qquad  q^{-p^{\ell(I)} f(e-1) s} \prod_{j=1}^J 	q^{(v_j -1) r_j -s f\klmm{ p^{R_j }- p^{ R_{j-1}}  } \klmm{  \tilde l_j(v_j) +1}  } 
\\ & \qquad  
\sum_{ \substack{  \mu  \in \mathcal C_I(\tau) 		\\ \omega_{k(\mu)}  = \omega \\ v_\mu = v } }
 \prod_{j=1}^J
 q^{ k_j \klmm{ (p-1) r_j -s pd\klmm{ p^{R_j }- p^{ R_{j-1}} } } }	  \\
=:&  \sum_{\omega \in \Omega_J} \sum_{v\in \{1,\ldots, p-1\}^J}c(\omega,v)\cdot h_{v}(s)  \cdot  \Psi_{\omega}(s) ,
	\end{align*}
where $\Psi_\omega(s)$ ranges over the terms containing $k_j$ and is meromorphic, $ h_{v}(s) $ is  a holomorphic function collecting the terms depending on $v_j$ and $ c(\omega,v)$ is a constant. 
We have $ \Psi_\omega(s)  = 
\sum\limits_{ \substack{  \mu  \in \mathcal C_I(\tau) 	\\ \omega_{k(\mu)}  = \omega \\ v_\mu = v } }
 \prod\limits_{j=1}^J
 q^{ k_j \klmm{ (p-1) r_j -s pd\klmm{ p^{R_j }- p^{ R_{j-1}} } } } $, 
$$	h_{v}(s)  =  q^{-p^{\ell(I)} f(e-1) s}
	\prod_{j=1}^J  	q^{ -s f \klmm{ p^{R_j }- p^{ R_{j-1}} } ( \tilde l_{j}(v_j) +1 ) }
	q^{r_j v_j}, $$
  \begin{equation} \label{c-omega-v} \displaystyle  c(\omega,v) := \widetilde c\left( \omega, v \right) \frac{p^{\kappa(I)}\prod\limits_{j=1}^J  \left( q^{r_j } -1 \right) q^{-r_j}}{ \Lambda\left(\omega , v \right)}. \end{equation}
  Concerning $\Psi_\omega(s)$,
  \[  
  \mu  \in \mathcal C_I(\tau): \substack{\omega_{k(\mu)}  = \omega , \\ v_\mu = v } \iff \substack{ \mu = (pek_1 + \tilde v_1, \ldots, pek_1 + \tilde v_{A_1(\omega)}, pek_2 + \tilde v_{A_1(\omega)+1}, \ldots pek_{\lambda(\omega)} + \tilde v_{J}) , \\
  k_1 < k_2 < \ldots < k_{\lambda(\omega)}, \quad    \tilde v_j = \tilde l_j(v_j).}
  \] The factors in $\Psi_\omega(s)$ are independent of $v_i$, hence we get
\begin{align*}
	\Psi_\omega(s) & = 
\sum\limits_{ \substack{  \mu =( k_1,\ldots, k_J)	\\ \omega_{k(\mu)}  = \omega } }
 \prod\limits_{j=1}^J
 q^{ k_j \klmm{ (p-1) r_j -s pd\klmm{ p^{R_j }- p^{ R_{j-1}} } } } \\
	 & \overset{\eqref{comp-Ajs}}{=} \sum_{ \left( k_1< \ldots < k_{\lambda(\omega)} \right) } 
	\prod_{j=1}^{\lambda(\omega)} q^{ k_j \sum\limits_{t=A_{j-1}(\omega)+1}^{A_j(\omega)} (p-1) \klmm{ r_t -spd \klmm{  p^{R_t} - p^{R_{t-1}}  }	} }.\qedhere
\end{align*}
\end{proof}

We can adapt the proof in \cite{Nicolas} for the asymptotics of elementary abelian extensions over local function fields and establish a decomposition of the meromorphic parts $\Psi_\omega(s)$ in the Dirichlet series. We mainly use Lemma~5.14  in \cite{Nicolas}:

\smallskip
\begin{lem}[Potthast] \label{Nicolas-La}
	Let X > 1 be real, and $\alpha_i(s) $ for $1\leq i \leq J$ be complex polynomials
	of degree 1 with real coefficients whose leading coefficients are negative. Then we have
	\[ 
	\Psi_{\alpha_1,\ldots, \alpha_J}(s):=	\sum_{k_1 \geq 0} \klmm{ X^{\alpha_1(s)}}^{k_1} \sum_{k_2 \geq 0 }^{k_1 - 1 }  
	\klmm{ X^{\alpha_2(s)}}^{k_2} \ldots 
	\sum_{k_J \geq 0 }^{k_{J-1} - 1 }  
	\klmm{ X^{\alpha_J(s)}}^{k_J} \]
	\[	= \frac{1}{ 1 - X^{\sum\limits_{j=1}^J  \alpha_j(s) }  } \cdot \prod_{i=1}^{J-1} \frac{ X^{\sum\limits_{j=1}^i \alpha_j(s)}   }{  1 - X^{\sum\limits_{j=1}^i \alpha_j(s)}   } .
	\]
\end{lem}

With this, we can attain the critical poles for Theorem~\ref{thm:Asymptotik-ZerfL}.
\smallskip
\begin{cor}\label{Psicor}
	Let $J := \cardI$ with $I$ and $\tau$ as before. Set  $S_j:= r_J + \ldots + r_{J-j+1}$, $R_j:= r_1+\ldots + r_j$ and 
	\begin{align*} 
		\alpha_i(s) &:=  \klmm{p-1} r_{J- i+1 }  - pd \klmm{ p^{R_{J- i+1} } - p^{R_{J-i}}  } s,
	\end{align*}	
	\[
	\Psi_{I, \tau}(s) := \Psi_{\alpha_1,\ldots, \alpha_{J}}(s) 
	\] as in Lemma~\ref{Nicolas-La}.
	
	Write  $\sigma_j (\tau) := \frac{(p-1)}{pd } \frac{  S_j   }{ p^{R_J } - p^{R_{J-j}} }
	= \frac{(p-1)}{pd } \frac{  S_j   }{ p^{R_{J-j}}(p^{S_j } - 1) } $ for $1\leq j \leq J$, then the only candidates for poles 
	of $\Psi_{I,\tau}(s)$ are  \[ \sigma_j(\tau) + \frac{2\pi i}{ pd\klmm{ p^{R_J } - p^{R_{J-j}} }\log(q)  } \mathbb Z + \frac{2 \pi i}{\log(q)} \mathbb Z \]
	and the maximal real-valued pole of $ \Psi_{I, \tau}(s) $ is $\sigma_{J}(\tau) = \frac{p-1}{pd } \frac{\ell(I)}{p^{\ell(I)} - 1} $.
\end{cor}
\begin{proof}
	For $1\leq j \leq J$ we get  using $R_J=\ell(I)$ and
	$ S_j:=r_J + \ldots + r_{J-j+1}$
	\begin{equation} \label{sum-alpha-i}
		\beta_j(s) :=	\sum_{i=1}^j \alpha_i(s) = {(p-1)  S_j  }  - pd\klmm{ p^{R_J } - p^{R_{J-j}} } s,
	\end{equation}
	By Lemma~\ref{Nicolas-La}, the corresponding meromorphic continuation of 
	\[ 
	\Psi_{I, \tau}(s) =	\sum_{k_1 \geq 0} \klmm{ q^{\alpha_1(s)}}^{k_1} \sum_{k_2 \geq 0 }^{k_1 - 1 }  
	\klmm{ q^{\alpha_2(s)}}^{k_2} \ldots 
	\sum_{k_J \geq 0 }^{k_{J-1} - 1 }  
	\klmm{ q^{\alpha_J(s)}}^{k_J} 
	\] 
	has candidates for poles at the roots of $\beta_j(s)$, i.e. on the vertical axes with fixed real part (Compare \cite[Corollary~5.19]{Nicolas}) through
	\[
	\sigma_j (\tau) := \frac{(p-1)}{pd } \frac{  S_j   }{ p^{R_J } - p^{R_{J-j}} }
	\]
	where all roots of $\beta_j(s)$ are \[ \sigma_j(\tau) + \frac{2\pi i}{pd\klmm{ p^{R_J } - p^{R_{J-j}} } \log(q) } \mathbb Z + \frac{2 \pi i}{\log(q)} \mathbb Z. \]
	Note that $r_j$, $R_j$, and $S_j$ are depending on the chosen $\tau$.
	Using 
	$S_{j+1}=S_j+r_{J-j}$ and $R_J=S_j+R_{J-j}$ for $1\leq j <J$ we get 
	\begin{align*}
		\sigma_j(\tau) < \sigma_{j+1}(\tau) &\iff  \frac{p^{R_{J-j-1}}(p^{S_{j+1} } - 1) }{p^{R_{J-j}}(p^{S_j } - 1) } < \frac{S_{j+1}  }{ S_j }\\
		&\iff \frac{p^{S_{j+1}} - 1}{p^{S_{j+1}} - p^{r_{J - j } }} < \frac{S_{j} + r_{J-j} }{ S_j }
		\\
		&\iff S_j p^{S_{j+1}} - S_j  < (p^{S_{j+1}} - p^{r_{J - j }}) (S_{j} + r_{J-j}) 
		\\
		&\iff (p^{r_{J - j }}-1) S_j  < p^{r_{J - j }}(p^{S_{j}} - 1) r_{J-j}. 
	\end{align*}
	
	Using $(p^{r_{J - j }}-1)< p^{r_{J - j }}$, $S_j \leq  p^{S_{j}} - 1$, and $r_{J-j}\geq 1$ the last statement is true when considering that $\Psi_{\alpha_1,\ldots, \alpha_J}(s)> 0$ for $s \in \mathbb R$.
	
	Hence, the maximal real-valued root is
	\begin{equation*}
		\max \klmmset{ \sigma_j(\tau) ~|~ 1\leq j \leq J} = \sigma_{J}(\tau) = \frac{p-1}{pd } \frac{\ell(I)}{p^{\ell(I)} - 1}.\qedhere
	\end{equation*} 
\end{proof} 

Note that the critical pole  $\sigma_J(\tau)$ is independent of the chosen bijection $\tau$ while the non-maximal
real-valued poles $\sigma_j (\tau)$ and conclusively the error terms are still depending on $\tau$.

\smallskip
\begin{thm}
	Let 	$\displaystyle \Phi_\tau(s) = \sum_{ \omega \in \Omega_{J } } \Psi_\omega(s) \sum_{ v \in \{ 1,\ldots, p-1\}^{J}  } c(\omega, v) h_v(s)$ as in Theorem~\ref{Phi-Tau-Zerlegung}. 
	 Then, 	$h_{v}\left(  \sigma_J(\tau) \right) \in \mathbb R_{>0}$ and $\Psi_\omega(s)$ is a meromorphic function with 
a	critical pole at $\displaystyle \sigma_J$.
\end{thm}
\begin{proof}
We want to apply Lemma~\ref{Nicolas-La} on \eqref{eq:Psi-omega-h}. 
As  $k_1,\ldots, k_{\lambda(\omega)}$ is increasingly ordered in \eqref{eq:Psi-omega-h} in contrast to the form the assumptions of Lemma~\ref{Nicolas-La} , we must revert the ordering by substituting $\alpha_t(s) := \widetilde \alpha_{J+1-t}(s)$ and accordingly $\alpha^\omega_j(s) = 
\sum\limits_{ t= A_{\lambda(\omega) -j} +1}^{A_{\lambda(\omega) +1-j}(\omega) } \alpha_{J+1-t}(s) $. This way, 
\[
\Psi_\omega(s)  = \sum_{ \left( k_1< \ldots < k_{\lambda(\omega)} \right) } 
	\prod_{j=1}^{\lambda(\omega)} q^{ k_j \sum\limits_{t=A_{j-1}(\omega)+1}^{A_j(\omega)} \widetilde \alpha_t(s)  }\]
\[	=\sum_{ \left( k_1> \ldots > k_{\lambda(\omega)} \right) } 
	\prod_{j=1}^{\lambda(\omega)} q^{ k_{\lambda(\omega)+1-j}  \sum\limits_{t=A_{ \lambda(\omega) -j}(\omega)+1}^{A_{ \lambda(\omega) +1-j}(\omega)} \alpha_t(s)  }
\]
	\[ 
\overset{\text{La}~\ref{Nicolas-La}}{= } \frac{1}{ 1 - X^{\sum\limits_{j=1}^{\lambda(\omega)} }  \alpha^\omega_j(s) }   \cdot \prod_{i=1}^{\lambda(\omega)-1} \frac{ X^{\sum\limits_{j=1}^i \alpha^\omega_j(s)}   }{  1 - X^{\sum\limits_{j=1}^i \alpha^\omega_j(s)} } .
	\]
	The candidates for poles of $\Psi_\omega(s)$ arise as roots of 
	$\beta^\omega_j(s) = \sum\limits_{i=1}^j \alpha^\omega_i(s)$. 
	Using $\beta_j(s)$ from \eqref{sum-alpha-i}, we clearly have \[ 
	\{ \beta_1^\omega(s), \ldots , \beta_{\lambda(\omega)}^\omega(s) \} \subset \{ \beta_1(s),\ldots, \beta_{J}(s) \}
	\]  and $\beta^\omega_{\lambda(\omega)}(s) = \beta_{J}(s)$. The critical maximal real-valued pole of $\Psi_\omega(s)$ is
	$ \displaystyle \sigma_{J}(\tau) = \frac{p-1}{pd } \frac{\ell(I)}{p^{\ell(I)} - 1} $, by adapting the proof of Corollary~\ref{Psicor}.
	
	It remains to show that  $\sigma_{J}(\tau) $ is a pole of the counting function \[\Psi(s) = \sum\limits_{ \omega \in \Omega_{J } } \Psi_\omega(s) \sum\limits_{ v \in \{ 1,\ldots, p-1\}^{J}  } c(\omega, v) h_v(s).\] 
	
	Formula \eqref{c-omega-v} shows that $ c(\omega,v) \geq 0$. Since $\widetilde c(\omega,v)=0 $ or $1$ we get that $c(\omega,v)=0$ if and only if $\widetilde c(\omega,v)=0 $.  The holomorphic function $h_{v}(s)$ satisfies 
	$ h_{v}( \sigma_J(\tau)) > 0$ as it is a real-valued polynomial in $q^{-s}$ with strictly positive real coefficients.
	
	Finally, we attain a critical non-zero summand through the composition $\omega^{(1)} =(1,1,\ldots, 1) \in \Omega_J$ as  $\widetilde c\left( \omega^{(1)} , v \right)  \neq 0$ by Remark~\ref{feasible}(b) for all vectors $v\in \{ 1,\ldots, p-1\}^{J} $. This shows that the candidate $\sigma_{J}(\tau) $ for the maximal pole is indeed a pole of $\Psi_\omega(s)$ for all  the counting function  concluding the proof.
\end{proof}


Hence, this shows the asymptotic behaviour claimed in Theorem~\ref{thm:Asymptotik-ZerfL}. 

%

\bibliographystyle{alpha}
\bibliography{mybib} 

\end{document}